\documentclass[12pt]{amsart}

\usepackage{amssymb,amsbsy}
\usepackage{latexsym,array}
\usepackage{verbatim,color}
\usepackage{fullpage,euscript}
\usepackage{xcolor}
\usepackage{tikz}
\usetikzlibrary{arrows,shapes,trees}
\usepackage[colorlinks=true,linkcolor=blue,urlcolor=my_color,citecolor=magenta]{hyperref}

\definecolor{my_color}{rgb}{0,0.5,0.5}
\definecolor{MIXT}{rgb}{0.8,0.5,0.2}
\definecolor{mixt}{rgb}{0.5,0.3,0.2}
\definecolor{sin}{rgb}{0,0.5,0.5}
\definecolor{darkblue}{rgb}{0,0.1,0.8}
\definecolor{redi}{rgb}{0.5,0,0.4}

\numberwithin{equation}{section}

\makeatletter
\@namedef{subjclassname@2020}{\textup{2020} Mathematics Subject Classification}
\makeatother

\input {cyracc.def}

\newtheorem{thm}{Theorem}[section]
\newtheorem{prop}[thm]{Proposition}
\newtheorem{lm}[thm]{Lemma}
\newtheorem{cl}[thm]{Corollary}

\newtheorem{qtn}{Question}
\newtheorem{prob}[qtn]{Problem}

\theoremstyle{remark}
\newtheorem{rmk}[thm]{Remark}
\newtheorem{ex}[thm]{Example} 

\theoremstyle{definition}
\newtheorem{df}{Definition}

\newcommand {\ah}{{\mathfrak a}}
\newcommand {\be}{{\mathfrak b}}
\newcommand {\ce}{{\mathfrak c}}

\newcommand {\g}{{\mathfrak g}}

\newcommand {\h}{{\mathfrak h}}
\newcommand {\tth}{{\tilde\h}}

\newcommand {\me}{{\mathfrak m}}

\newcommand {\p}{{\mathfrak p}}
\newcommand {\q}{{\mathfrak q}}
\newcommand {\rr}{{\mathfrak r}}
\newcommand {\trr}{{\tilde\rr}}
\newcommand {\es}{{\mathfrak s}}
\newcommand {\te}{{\mathfrak t}}
\newcommand {\ut}{{\mathfrak u}}

\newcommand {\z}{{\mathfrak z}}

\newcommand {\glN}{{\mathfrak{gl}}_N}
\newcommand {\sln}{{\mathfrak{sl}}_n}

\newcommand {\slno}{{\mathfrak{sl}}_{n+1}}

\newcommand {\sone}{{\mathfrak{so}}_{2n}}

\newcommand{\gt}{\mathfrak}
\newcommand{\ttP}{{\tt P}}

\newcommand {\eus}{\EuScript}

\newcommand {\gA}{{\eus A}}

\newcommand {\gS}{{\eus S}}
\newcommand {\gZ}{{\eus Z}}

\newcommand {\esi}{\varepsilon}

\newcommand {\ap}{\alpha}

\newcommand {\ca}{{\mathcal A}}

\newcommand {\cF}{{\mathcal F}}

\newcommand {\cL}{{\mathcal L}}

\newcommand {\cz}{{\mathcal Z}}
\newcommand {\BA}{{\mathbb A}}
\newcommand {\BC}{{\mathbb C}}
\newcommand {\BP}{{\mathbb P}}
\newcommand {\BS}{{\mathbb S}}

\newcommand {\BZ}{{\mathbb Z}}
\newcommand {\BN}{{\mathbb N}}

\newcommand {\BQ}{{\mathbb Q}}

\newcommand {\ad}{{\mathrm{ad}}}

\newcommand {\Ann}{{\mathrm{Ann}}}

\newcommand {\codim}{{\mathrm{codim\,}}}

\newcommand {\gr}{{\mathrm{gr\,}}}

\newcommand {\ind}{{\mathrm{ind\,}}}
\newcommand {\Lie}{{\mathsf{Lie\,}}}

\newcommand {\Ima}{{\mathrm{Im\,}}}

\newcommand {\rk}{{\mathsf{rk\,}}}

\newcommand {\spe}{{\mathsf{Spec\,}}}

\newcommand {\tr}{{\mathrm{tr\,}}}
\newcommand {\trdeg}{{\mathrm{tr.deg\,}}}

\newcommand {\diag}{{\mathsf{diag}}}

\newcommand {\tri}{\mathfrak{sl}_2}

\newcommand {\bb}{{\boldsymbol{b}}}

\newcommand {\bt}{\boldsymbol{p}}
\newcommand {\fb}{\boldsymbol{f}}

\newcommand {\PC}{{\sf PC}}

\newcommand {\beq}{\begin{equation}}
\newcommand {\eeq}{\end{equation}}

\renewcommand{\le}{\leqslant}
\renewcommand{\ge}{\geqslant}
\renewcommand{\lg}{\langle}
\newcommand{\rg}{\rangle}
\newcommand{\ggs}{{\sf g.g.s.}}
\newcommand{\wrt}{{w.r.t.}}

\newcommand {\bbk}{\Bbbk}

\begin{document}
\setlength{\parskip}{3pt plus 2pt minus 0pt}
\hfill { {\color{blue}\scriptsize March 5, 2026}}
\vskip1ex

\title[Horospherical splittings]
{Horospherical splittings of $\g$ and related Poisson commutative subalgebras of $\gS(\g)$} 
\author[D.\,Panyushev]{Dmitri I. Panyushev}
\address[D.P.]{Independent University of Moscow,
119002 Moscow, Russia}
\email{panyush@mccme.ru}
\author[O.\,Yakimova]{Oksana S.~Yakimova}
\address[O.Y.]{Institut f\"ur Mathematik, Friedrich-Schiller-Universit\"at Jena,  07737 Jena,
Germany}
\email{oksana.yakimova@uni-jena.de}
\begin{abstract}
Let a Lie algebra $\q$ be a linear sum of two complementary subalgebras $\h$ and $\rr$. We continue 
our investigations initiated in ({\it J. London Math. Soc.} {\bf 103}\,(2021), 1577--1595), where compatible
Poisson brackets associated with splitting $\q=\h\oplus\rr$ and related 
Poisson-commutative subalgebras of the symmetric algebra $\gS(\q)$ are studied. In this article, we 
further develop the general theory and study in more details splittings of the reductive Lie algebras such 
that both $\h$ and $\rr$ are solvable horospherical subalgebras. We also derive some results of the
Adler--Kostant--Symes theory using our approach.
\end{abstract}
\maketitle

\section{Introduction}

\noindent
The ground field $\bbk$ is algebraically closed and  $\mathsf{char}\,\bbk=0$. All Lie algebras
are assumed to be algebraic. 
A {\it splitting} of  a Lie algebra $\q$ is a vector space decomposition $\q=\h\oplus\rr$, where $\h$ and 
$\rr$ are Lie subalgebras. Although this setup is similar to that of the Adler--Kostant--Symes 
Theorems~\cite[4.4]{AMV}, we pursue a different path. This paper is a sequel to~\cite{bn}, where we 
defined {compatible Poisson brackets} associated with splittings (called $2$-splittings in \cite{bn}) and 
studied related {\it Poisson-commutative\/} (=\PC) subalgebras $\gZ_{\lg\h,\rr\rg}$ of the symmetric 
algebra $\gS(\q)$. Our goal is to elaborate on good cases in which $\trdeg\gZ_{\lg\h,\rr\rg}$ is maximal 
and $\gZ_{\lg\h,\rr\rg}$ is a polynomial ring. This provides a number of new completely integrable 
systems. Basic constructions and results on $\gZ_{\lg\h,\rr\rg}$ are valid for arbitrary $\q$, but substantial 
applications are obtained for reductive algebraic groups and their Lie algebras. Therefore, in the rest of 
Introduction we stick to the reductive case. 

Let $G$ be a connected reductive group with $\Lie G=\g$.
If $\gA$ is a \PC\ subalgebra of  $\gS(\g)$, then $\trdeg\gA\le \bb(\g):=\frac{1}{2}(\dim\g+\rk\g)$, 
see~\cite[0.2]{v:sc}. We say that $\gA$ is {\it large}, if $\trdeg\gA= \bb(\g)$. The {\it algebra of symmetric 
invariants\/} of $\g$, denoted $\gS(\g)^\g$ or $\cz\gS(\g)$, is  the centre of the Lie--Poisson algebra 
$(\gS(\g), \{\,\,,\,\})$. It is a graded polynomial ring in $\ell=\rk\g$ variables, which we use for 
constructing \PC\ subalgebras of $\gS(\g)$. Any set of algebraically independent homogeneous 
generators of $\gS(\g)^\g$ is called a  {\it Hilbert basis}. If $\g=\h\oplus\rr$, then there are Lie algebra
contractions of $\g$ to the semi-direct products $\g_{(0)}{:}=\h\ltimes\rr^{\rm ab}$ and 
$\g_{(\infty)}{:}=\rr\ltimes \h^{\rm ab}$, where the second factor is an abelian ideal. Here $\rr^{\rm ab}$ (resp. $\h^{\rm ab}$) is naturally regarded as  
$\h$-module $\g/\h$ (resp. $\rr$-module $\g/\rr$). Let $\{\,\,,\,\}_{0}$ and $\{\,\,,\,\}_{\infty}$ be the Lie--Poisson brackets of $\g_{(0)}$ and $\g_{(\infty)}$, respectively. Then
$\{\,\,,\,\}_{0}+\{\,\,,\,\}_{\infty}$ is the 
Lie--Poisson bracket of $\g$ (cf. Section~\ref{sect:PC}). Hence these brackets are compatible. 
The subalgebra $\gZ_{\lg\h,\rr\rg}$ is obtained from these compatible Poisson brackets via the use 
of the {\it Lenard--Magri scheme} (cf. Section~\ref{subs:LM}), and the point is to 
realise whether $\gZ_{\lg\h,\rr\rg}$  
is large and enjoys other good properties. 
 
Let $\ind\q$ denote the {\it index} of $\q$. By the standard property of Lie algebra 
contractions, $\ind\g_{(0)}\ge \ind\g=\rk\g$ and $\ind\g_{(\infty)}\ge\rk\g$.
Then the following is true. 
\begin{itemize}
   \item[\sf (i)] \  $\ind\g_{(0)}=\rk\g$ if and only if $\h$ is a {\it spherical subalgebra} of $\g$ (and likewise 
for $\g_{(\infty)}$ and $\rr$), see~\cite[Section\,2]{bn} and~\cite{T}.
  \item[\sf (ii)] \ If  $\ind\g_{(0)}=\ind\g_{(\infty)}=\rk\g$, then 
$\trdeg\gZ_{\lg\h,\rr\rg}=\bb(\g)$~\cite[Theorem\,3.2]{bn}.
  \item[\sf (iii)] \ If ${\mathcal C}$ is an arbitrary large \PC\ subalgebra of $\gS(\g)$, then it
is complete on {\bf generic} regular $G$-orbits in $\g^*$~\cite[Lemma\,1.1]{bn}.
\end{itemize}
 
A splitting of $\g$ is said to be {\it non-degenerate}, if $\ind\g_{(0)}=\ind\g_{(\infty)}=\rk\g$.
In~\cite[Sect.\,4-6]{bn}, we found several important cases with large {\bf polynomial} rings 
$\gZ_{\lg\h,\rr\rg}$. For instance,

\textbullet \ \ If $\be$ and $\be_-$ are opposite Borel subalgebras of $\g$ and $\ut_-=[\be_-,\be_-]$, then 
$\g =\be\oplus\ut_-$ is a non-degenerate splitting. Here $\gZ_{\langle\be,\ut_-\rangle}$ is also a 
{\it maximal\/} \PC\ subalgebra of $\gS(\g)$, which is complete on {\bf all} regular coadjoint orbits of $G$. 
\\ \indent
\textbullet \ \ Let $\vartheta^{\sf max}$ be an involution of $\g$ of {\it maximal rank}, i.e., its $(-1)$-eigenspace  
contains a Cartan subalgebra. If $\g_0$ is the fixed-point subalgebra of 
$\vartheta^{\sf max}$, then there is a Borel subalgebra $\be$ such that $\g=\be\oplus\g_0$. This 
splitting is non-degenerate and, at least for $\g=\sln$, $\gZ_{\langle\be,\g_0\rangle}$ is also a
{\it maximal\/} \PC\ subalgebra. 

The sum $\g=\h\oplus\rr$ provides a bi-homogeneous decomposition for every homogeneous 
$F\in\gS(\g)$. Let $F_1,\dots,F_\ell$ be a Hilbert basis of $\gS(\g)^\g$. By~\cite{bn}, $\gZ_{\lg\h,\rr\rg}$ is
generated by the bi-homogeneous components of all $\{F_j\}$ and the generators of 
$\cz\gS(\g_{(0)})$ and $\cz\gS(\g_{(\infty)})$.
Let $F^\bullet$ (resp. $F_\bullet$) denote the nonzero bi-homogeneous component of $F$ with minimal 
$\h$-degree (resp. $\rr$-degree). Then $F^\bullet\in\cz_0:=\cz\gS(\g_{(0)})$ and 
$F_\bullet\in\cz_\infty:=\cz\gS(\g_{(\infty)})$. Following~\cite{oy,bn}, we say that $\{F_j\}_{j=1}^\ell$ is a 
{\it good generating system} (=\,\ggs) {\it for\/} $\h$, if $F_1^\bullet,\dots,F_\ell^\bullet$ are algebraically 
independent (and likewise for $\rr$ and $\{F_{j\bullet}\}$). The presence of a \ggs\ is usually required for 
proving that $\cz_0$, $\cz_\infty$, and $\gZ_{\lg\h,\rr\rg}$ are polynomial rings.

In this article, we further develop the general theory and describe new instances of non-degenerate
splittings with large {\bf polynomial} rings $\gZ_{\lg\h,\rr\rg}$. 

{\bf 1}) \  Let $H$ and $R$ be the connected subgroups of $G$ with Lie algebras $\h$ and $\rr$, 
respectively. For a splitting $\g=\h\oplus\rr$, set  $s_0:=\trdeg \bbk((\g/\h)^*)^{H}$ and 
$s_\infty:=\trdeg \bbk((\g/\rr)^*)^{R}$. We prove that $s_0+s_\infty\ge\ell$, and if $s_0+s_\infty=\ell$, then 
this splitting  is non-degenerate (Theorem~\ref{sum-r}). For a non-degenerate splitting, it is shown that if 
$\cz_0$ and $\cz_\infty$ are polynomial rings and there is a common \ggs\ for $\h$ and $\rr$, then 
$\gZ_{\lg\h,\rr\rg}$ is a polynomial ring (Theorem~\ref{thm:ggs2}). Another result is that if $\cz_0$ and 
$\cz_\infty$ are polynomial rings and $s_0+s_\infty=\ell$, then $\gZ_{\lg\h,\rr\rg}$ is a polynomial ring 
(Theorem~\ref{thm:sum}). We also relate the sum $s_0+s_\infty$ to some properties of the 
representations $(H:(\g/\h)^*)$ and $(R:(\g/\rr)^*)$, cf. Section~\ref{subs:pro-r-and-s}.

{\bf 2}) \  Let $\h_+$ be a {solvable horospherical} subalgebra of $\g$, i.e., 
$\ut\subset\h_+\subset\be$ for some Borel $\be\subset\g$. Then it has a 
complementary horospherical subalgebra $\h_-\subset\be_-$. The resulting non-degenerate splitting 
$\g=\h_+\oplus\h_-$ is said to be {\it horospherical}. We give a criterion for existence of a \ggs\ for 
$\h_+$ and show that if there is a \ggs, then $\cz\gS(\h_+\ltimes \h_-^{\sf ab})$ is a polynomial ring. 
Therefore, if there is a \ggs\ for $\h_+$ and $\h_-$ 
(not necessarily a common one!), then $\gZ_{\lg\h_+,\h_-\rg}$ is a (large) polynomial ring, see 
Section~\ref{sect:horo}. Using the toral part of $\h_+$ and a classical result of Richardson~\cite{Ri},
we also provide a useful necessary condition for presence of \ggs

{\bf 3}) \  An nice horospherical splitting occurs if $\te$ is a Cartan 
subalgebra of $\g$ and $\tilde\g=\g\dotplus\te$. Then $\tilde\g$ can be written as a direct sum of two 
horospherical subalgebras $\tilde\h$ and  $\tilde\rr$ such that either of them is isomorphic to a Borel 
subalgebra of $\g$. Hence the contractions $\tilde\g_{(0)}$ and $\tilde\g_{(\infty)}$ are 
isomorphic.  
We point out certain \ggs\ for $\tilde\h$ and $\tilde\rr$, 
which  implies that $\cz\gS(\tilde\g_{(0)})$ is a polynomial ring and
$\gZ_{\lg\tilde\h,\tilde\rr\rg}$ is a (large) polynomial \PC\ subalgebra of $\gS(\tilde\g)$
(Theorem~\ref{thm:bb}). Moreover, if the degrees of $F_1,\dots,F_\ell$ are even, then we get a {\sl common} \ggs\ for $\tilde\h$ and $\tilde\rr$.

{\bf 4}) \ The $(-1)$-eigenspace of an involution of maximal rank contains regular elements of $\g$. If $\g$ is simple, then there are few other cases with this property. The corresponding pairs  $(\g,\g_0)$ are: 
$(\mathfrak{sl}_n, \mathfrak{sl}_m\dotplus\mathfrak{sl}_{n-m}\dotplus\bbk)$ with $|n-2m|\le 1$; 
$(\mathfrak{so}_{2n}, \mathfrak{so}_{n+1}\dotplus \mathfrak{so}_{n-1})$;
$(\eus{E}_{6}, \eus{A}_{5}\dotplus\eus{A}_{1})$. For these $\g_0$, there is  a complementary solvable 
horospherical subalgebra $\h\subset\g$. Furthermore, $\g_0$ admits a \ggs\ and 
$\cz\gS(\g_0\ltimes\h^{\sf ab})$ is a polynomial ring~\cite{coadj, contr}. We prove that $\h$ admits a \ggs\ 
exactly if $\g=\mathfrak{sl}_{2n}$ or $\sone$. This implies that $\gZ_{\lg\h,\g_0\rg}$ is a large polynomial 
ring in these cases, see Section~\ref{sect:ohne}.

{\bf 5}) \ In Section~\ref{sect:related}, we show that our methods allow to derive some old results related 
to the Adler--Kostant--Symes theory and point out directions for further research.

{\bf Some notation}.  \\
{\bf --} \  A direct sum of Lie algebras is denoted by `$\dotplus$'; \\
{\bf --} \  $\ind\q$ is the index of a Lie algebra $\q$  and $\bb(\q)=\frac{1}{2}(\dim\q+\ind\q)$;      \\
{\bf --} \   $\cz\gS(\q)$ or $\gS(\q)^\q$ is the centre of the Lie--Poisson algebra $(\gS(\q), \{\,\,,\,\})$;  \\
{\bf --} \  $\bbk[X]$ is the algebra of regular functions on an algebraic variety $X$;  \\
{\bf --} \  $\bbk(X)$ is the field of rational functions on an irreducible algebraic variety $X$.

\section{Preliminaries on Poisson brackets and contractions}
\label{sect:prelim}

\noindent
Let $Q$\/ be an affine algebraic group with Lie algebra $\q$. Write $\gS(\q)$ for the symmetric algebra 
of $\q$ over $\bbk$. It is naturally identified with the graded algebra of polynomial functions on $\q^*$,  
also denoted by $\bbk[\q^*]$. The Lie--Poisson bracket in $\gS(\q)$ is defined on the elements of degree 
$1$ (i.e., on $\q$) by $\{x,y\} :=[x,y]$. This makes $\q^*$ a Poisson variety, and $\pi$ denotes the 
corresponding Poisson tensor, see e.g.~\cite[Chap.\,1.2]{duzu}. 

For $\xi\in\q^*$, let $Q\xi$ denote its coadjoint orbit. Then $Q^\xi$ is the stabiliser of $\xi$ in $Q$ and
$\q^\xi=\Lie(Q^\xi)\subset\q$. 
The {\it index of}\/ $\q$, $\ind\q$, is the minimal codimension of $Q$-orbits in $\q^*$, i.e., 
$\ind\q=\min_{\xi\in\q^*} \dim \q^\xi$. By Rosenlicht's theorem~\cite[\S\,2.3]{VP}, one also has
$\ind\q=\trdeg \bbk(\q^*)^Q$.
The set of $Q$-{\it regular\/} elements of $\q^*$ is
\beq       \label{eq:regul-set}
    \q^*_{\sf reg}=\{\eta\in\q^*\mid \dim \q^\eta=\ind\q\} 
\eeq
and $\q^*_{\sf sing}=\q^*\setminus \q^*_{\sf reg}$. We say that $\q$ has the {\sl codim}--$2$ property if 
$\codim \q^*_{\sf sing}\ge 2$.
The {\it Poisson centre\/} of $\gS(\q)$ is
\[
    \cz\gS(\q):=\{F\in \gS(\q)\mid \{F,x\} =0 \ \ \forall x\in\q\}.
\]
It follows that $\cz\gS(\q)$ coincides with the {\it algebra of symmetric invariants\/} of $\q$, also denoted 
$\gS(\q)^\q$. If $Q$ is connected, then $\cz\gS(\q)=\gS(\q)^\q=\bbk[\q^*]^Q$ and the symplectic leaves in 
$\q^*$ are the coadjoint $Q$-orbits. For any 
$\xi\in\q^*$, $\pi(\xi)$ is a skew-symmetric bilinear form on $T^*_\xi(\q^*)\simeq \q$, which is defined by 
$\pi(\xi)(x,y)=\xi([x,y])$ for $x,y\in\q$. It is easily seen that $\ker\pi(\xi)=\q^\xi$. Hence the restriction of 
$\pi(\xi)$ to $T^*_\xi(Q\xi)\simeq \q/\q^\xi$ is non-degenerate. 

If $\gA$ is a \PC\ subalgebra of $\gS(\q)$, then $\{\gA|_{Q\xi},\gA|_{Q\xi}\}=0$. Hence 
$\trdeg(\gA|_{Q\xi})\le\frac{1}{2}\dim Q\xi$. Then $\gA$ is 
said to be {\it complete on\/} $Q\xi$, if $\trdeg(\gA|_{Q\xi})=\frac{1}{2}\dim Q\xi$.
The passage from $\q^*$ to $Q\xi$ shows that  $\trdeg\gA\le \bb(\q):=\frac{1}{2}(\dim\g+\ind\q)$.
Since the coadjoint orbits are even-dimensional, $\bb(\q)$ is an integer. If $\q$ is reductive, then
$\ind\q=\rk\q$ and $\bb(\q)$ equals the dimension of a Borel subalgebra of $\q$. 

\subsection{The Lenard--Magri scheme} 
\label{subs:LM} 
Let $\{\ ,\,\}'$ and $\{\ ,\,\}''$ be compatible Poisson brackets on $\q^*$~\cite[1.8.3]{duzu}. This yields a 
$2$-parameter family of Poisson brackets $a\{\ ,\,\}'+b\{\ ,\,\}''$, $a,b\in\bbk$, and we assume that
this family contains the initial Lie--Poisson bracket. As we are mostly interested in the 
corresponding Poisson centres, it is convenient to organise this, up to scaling, in a 1-parameter family 
$\{\ ,\ \}_t=\{\ ,\ \}'+t\{\ ,\ \}''$, $t\in\BP:=\bbk\cup\{\infty\}$, where $t=\infty$ corresponds to  
$\{\ ,\ \}''$. Let $\q^*_{(t)}$ denote the Poisson variety corresponding to $t\in\BP$.
Let $\ind\!\{\,\,,\,\}_t$ denote is the minimal codimension of the symplectic leaves in $\q^*_{(t)}$.  
For almost all $t\in\BP$, $\ind\{\,\,,\,\}_t$ has one and the same (minimal) value. Set
\[ 
    \text{ $\BP_{\sf reg}=\{t\in \BP\mid \ind\!\{\,\,,\,\}_t \text{ is minimal}\}$ \ and \ 
   $\BP_{\sf sing}=\BP\setminus \BP_{\sf reg}$. }
\]
Let $\cz_t$ denote the centre of the Poisson algebra $(\gS(\q),\{\ ,\ \}_t)$. The crucial fact is that the algebra $\gZ$ generated by 
$\{\cz_t\mid t\in\BP_{\sf reg}\}$ is Poisson-commutative w.r.t{.} to {\bf any} bracket in the family.
In many cases, this procedure provides a \PC\ subalgebra of $\gS(\q)$ of maximal transcendence 
degree. An obvious first step is to take the initial Lie--Poisson 
bracket $\{\,\,,\,\}$ as $\{\,\,,\,\}'$. The rest depends on a clever choice of $\{\,\,,\,\}''$.

If $\q=\g$ is reductive, then a notable realisation of this scheme is the {\it argument shift method}, which 
goes back to~\cite{mf}. Here $\{\,\,,\,\}''=\{\,\,,\,\}_\gamma$ is the Poisson bracket of degree zero 
associated with $\gamma\in\g^*$, where $\{\xi,\eta\}_\gamma:=\gamma([\xi,\eta])$ for $\xi,\eta\in\g$. The 
\PC\ subalgebras $\gZ=\gZ_\gamma$ occurring in this way are known as {\it Mishchenko--Fomenko subalgebras}. 
In this setting, a ``clever'' choice means that $\gamma$ should be a regular element of $\g^*$.

In this article, we consider only families consisting of {\bf linear} Poisson brackets, i.e., the brackets 
$\{\ ,\ \}_t$ ($t\in\BP$) provide various Lie algebra structures on the underlying vector space $\q$.
The Lie algebra corresponding to $t$ is denoted by $\q_{(t)}$ and then $\ind\q_{(t)}=\ind\{\,\,,\,\}_t$.

\subsection{Contractions and invariants} 
\label{subs:contr-&-inv} 
Let $\h$ be a subalgebra of $\q$. Consider the semi-direct product $\q_{(0)}=\h\ltimes (\q/\h)^{\rm ab}$, 
where the superscript `ab' means that the $\h$-module $\q/\h$ is an abelian ideal in $\q_{(0)}$. If 
$\me\subset\q$ is any complementary subspace to $\h$, then we can also write 
$\q_{(0)}=\h\ltimes\me^{\rm ab}$, where the new Lie bracket $[\ ,\, ]_0$ on the vector space $\q=\q_{(0)}$ 
is defined by
\[
         [h_1+m_1,h_2+m_2]_0=[h_1,h_2]+ {\sf pr}_\me([h_1,m_2]- [h_2,m_1]),
\]
where $h_i\in\h$, $m_i\in\me$, and ${\sf pr}_\me: \q\to \me$ is the projection parallel to $\h$. The Lie 
algebra $\q_{(0)}$ is obtained from $\q$ via a contraction, and it is called an {\it In\"on\"u--Wigner 
contraction} of $\q$, see e.g.~\cite[Ch.\,7, \S\,2.5]{t41},\cite{alafe},\cite{T}. A lot of information on the 
symmetric invariants of In\"on\"u--Wigner contractions is obtained in~\cite{coadj, rims, contr, Y-imrn}. 

The sum $\q=\h\oplus \me$ provides the bi-homogeneous decomposition of any homogeneous 
polynomial $F\in \gS(\q)$:
\[
    \textstyle F=\sum_{i=0}^{d} F_{i,d-i} \ .
\]
Here $d=\deg F$ and $F_{i,d-i}\in \gS^i(\h)\otimes \gS^{d-i}(\me)\subset \gS^{d}(\q)$. Then $(i,d-i)$ is 
the {\it bi-degree\/} of $F_{i,d-i}$. We also set $\deg_\h F_{i,d-i}=i$ and $\deg_\me F_{i,d-i}=d-i$. Let 
$F^\bullet$ (resp. $F_{\bullet}$) denote the nonzero bi-homogeneous component of $F$ with maximal 
$\me$-degree (resp. $\h$-degree). Then we set $\deg_{\me} F:=\deg_{\me} F^\bullet$ and 
$\deg_{\h} F:=\deg_{\h} F_\bullet$. If $F\in\gS(\q)^\q$, then $F^\bullet$ is a symmetric invariant of 
$\q_{(0)}$, i.e., $F^\bullet\in \gS(\q_{(0)})^{\q_{(0)}}$~\cite[Prop.\,3.1]{coadj}. 
\\ 
\indent The following observation shows that the properties of highest components does not 
essentially depend on the choice of a complementary subspace $\me$. Write temporarily $F^\bullet_\me$
for the highest component of $F$ with respect to $\me$.

\begin{lm}   
\label{lm:independ}
Let\/ $\me$ and\/ $\tilde\me$ be different complementary subspaces to $\h$. Then there is a linear 
operator $\cL: \q\to \q$ such that $\cL\vert_\h={\sf id}$, $\cL(\me)=\tilde\me$, and 
$\cL(F^\bullet_\me)=F^\bullet_{\tilde\me}$ \ for all $F\in\gS(\q)$. 
\emph{(Here we use the natural extension of $\cL$ to $\gS(\q)$.)}
\end{lm}
\begin{proof}
Let $(y_j)_{j\in J}$ be a basis for $\me$. For each $j\in J$, we have $y_j=\tilde y_j+z_j$ with some 
$\tilde y_j\in\tilde\me$ and $z_j\in\h$. Then $(\tilde y_j)_{j\in J}$ is a basis for $\tilde\me$, and we set 
$\cL(y_j)=\tilde y_j$. To prove the last relation, it suffices to treat the case in which $F$ is a monomial.
Suppose that $F=x_1^{a_1}\dots x_s^{a_s}{\cdot} y_1^{b_1}\dots y_t^{b_t}$, where $\{x_i\}$ is a basis 
for $\h$. Then $F=F^\bullet_\me$. In terms of $\tilde\me$, we have
$F=x_1^{a_1}\dots x_s^{a_s}{\cdot} (\tilde y_1+z_j)^{b_1}\dots (\tilde y_t+z_t)^{b_t}$ and hence
$F^\bullet_{\tilde\me}=x_1^{a_1}\dots x_s^{a_s}{\cdot}\tilde y_1^{b_1}\dots\tilde y_t^{b_t}=\cL(F^\bullet_\me)$.
\end{proof}

The quotient field of $\gS(\q)^\q$, $\mathsf Q(\gS(\q)^\q)$, is contained in $\bbk(\q^*)^Q$. Hence 
$\trdeg \gS(\q)^\q\le \ind\q$.

\begin{df}          \label{def-ggs}
Suppose that $\gS(\q)^\q$ is a graded polynomial ring.
A Hilbert basis $F_1,\dots,F_\ell$ in $\gS(\q)^\q$ is called a {\it good generating system} (=\ggs)
{\it for\/} $\h$, if $\ell=\ind\q$ and $F_1^\bullet,\dots,F_\ell^\bullet$ are algebraically independent. 
We also say that $\h$ {\it admits a\/} {g.g.s.}
\end{df}

The assumptions of this definition are satisfied if $\q$ is reductive.
By Lemma~\ref{lm:independ}, the notion of a \ggs\ is well defined.

\begin{ex}    \label{ex:ggs}
1. The property of being `good' depends on a generating system (i.e., Hilbert basis). For $\g=\glN$, there 
are two well-known Hilbert bases in $\gS(\g)^\g\simeq\bbk[\glN]^{GL_N}$. The first basis 
consists of the coefficients of the characteristic polynomial of a matrix $A\in\glN$. The second basis is formed by the traces of powers of $A$, i.e.,
$F_i(A)=\tr(A^i)$, $i=1,\dots,N$. By~\cite[Theorem\,4.5]{coadj}, the first basis is a {\sf g.g.s.} 
for $\h=\mathfrak{gl}_{N-m}\dotplus\mathfrak{gl}_m$. But the second basis is not a {\sf g.g.s.}
unless $|N-2m|\le 1$. This can be checked directly or using Theorem~\ref{thm:kot14}{\sf (ii)}.
\\  \indent
2. It can also happen that, for a given $\h$, there is no {\sf g.g.s.} at all,  cf.~\cite[Remark\,4.3]{coadj} and 
examples in Sections~\ref{subs:gl2} and \ref{subs:E6}. 
\end{ex}

The following important result shows that, for a wide class of Lie algebras, 
 {\bf (1)} there is an effective method to verify that a Hilbert basis $(F_1,\dots,F_\ell)$ is a \ggs\ 
for $\h$; 
 {\bf (2)} the presence of a \ggs\ is helpful in describing the Poisson centre for the In\"on\"u--Wigner 
contraction $\q_{(0)}=\h\ltimes\me^{\rm ab}$. 

\begin{thm}[{\cite[Theorem\,3.8]{contr}}]         \label{thm:kot14}
Let $\q$ satisfy the following properties: ({\it\bfseries a})  $\q$ has the {\sl codim}-$2$ property,
({\it\bfseries b}) $\gS(\q)^\q$ is a polynomial ring in $\ell=\ind\q$ variables, and ({\it\bfseries c}) for a (=any) Hilbert basis of $\gS(\q)^\q$, one has $\sum_{i=1}^\ell\deg F_i=\bb(\q)$.
Then
\begin{itemize}
\item[\sf (i)] \ $\sum_{j=1}^{\ell} \deg_{\me}\! F_j\ge \dim\me$;
\item[\sf (ii)] \ $F_1,\dots,F_{\ell}$ is a {\sf g.g.s.} for $\h$ if and only if\/ $\sum_{j=1}^{\ell} \deg_{\me}\! F_j=\dim\me$;
\item[\sf (iii)] \ if\/ $\q_{(0)}$ has the {\sl codim}--$2$ property, $\ind\q_{(0)}=\ell$, and  
$F_1,\dots,F_{\ell}$ is a {\sf g.g.s.} for\/ $\h$, then $\cz\gS(\q_{(0)})$ is a polynomial ring that is freely 
generated by $F_1^\bullet,\dots,F_{\ell}^\bullet$. 
\end{itemize}
\end{thm}
\begin{rmk}
1) The assumptions {\it\bfseries (a),(b),(c)} hold for the reductive Lie algebras $\g$, where $\ind\g=\rk\g$
and $\bb(\g)=\dim\be$. \\ \indent
2) Since $\q_{(0)}$ is a contraction of $\q$, $\ind\q_{(0)}\ge \ind\q=\ell$. Hence a condition in {\sf (iii)} means that $\ind\q_{(0)}$ is as small as possible.
\\ \indent
3) If the {\sl codim}--$2$ property does not hold in {\sf (iii)}, then it can happen that 
$\bbk[F_1^\bullet,\dots,F_\ell^\bullet]$ is a proper subalgebra of $\cz\gS(\q_{(0)})$. On the other hand, 
the equality $\bbk[F_1^\bullet,\dots,F_\ell^\bullet]=\cz\gS(\q_{(0)})$ may hold  without the 
{\sl codim}--$2$ property for $\q_{(0)}$, cf. examples in Section~\ref{sec:g+t}.
\end{rmk}
Let $\cz\gS(\q)^\bullet$ denote the subalgebra of $\gS(\q)$ generated by all $F^\bullet$, where 
$F\in \cz\gS(\q)$ is homogeneous. Equivalently, if $\cz\gS(\q)_p$ is the vector space of symmetric invariants of degree $p$ and $\cz\gS(\q)^\bullet_p=\{F^\bullet\mid F\in\cz\gS(\q)_p\}$, then
$\cz\gS(\q)^\bullet=\bigoplus_{p\ge 0}\cz\gS(\q)^\bullet_p$.

\begin{lm}           \label{bullet}
Regardless of the presence of a \ggs, we have  $\trdeg \cz\gS(\q)^\bullet=\trdeg \cz\gS(\q)$.
\end{lm}
\begin{proof}
Using rudiments of Gr\"obner bases technique, one can construct a basis $(H_1,\dots,H_n)$ for 
$\cz\gS(\q)_p$ such that $H_1^\bullet,\dots, H_n^\bullet$ is a basis for $\cz\gS(\q)^\bullet_p$, cf.~\cite[Lemma\,1.3]{alafe}.
Therefore, the Poincar{\'e} series of the graded algebras $\cz\gS(\q)$ and $\cz\gS(\q)^\bullet$ 
coincide. Hence they have one and the same transcendence degree by~\cite[Satz\,4.5]{bk-GK}. 
\end{proof}

\section{Poisson-commutative subalgebras related to splittings}
\label{sect:PC}

Let $\q=\h\oplus\rr$ be a splitting of $\q$. Set $\q_{(0)}:=\h\ltimes\rr^{\rm ab}$ and 
$\q_{(\infty)}:=\rr\ltimes \h^{\rm ab}$. We identify $\q,\q_{(0)},$ and $\q_{(\infty)}$ as vector spaces.
Let $x_\rr\in\rr$ and $x_{\h}\in\h$ denote the components of  $x\in\q$. The semi-direct structure of 
$\q_{(0)}$ and $\q_{(\infty)}$ shows that the corresponding Poisson brackets are defined on 
$\q=\gS^1(\q)$ as follows:
\beq    \label{eq:2-brak}
\{x,y\}_0=\begin{cases}  [x,y] & \text{ if } \ x,y\in \h, \\
{}[x,y]_{\rr} & \text{ if } \ x \in\h,y\in \rr, \\
\quad 0 & \text{ if } \ x,y\in \rr ,
\end{cases} \ \ \& \ \   
\{x,y\}_\infty=\begin{cases}\quad 0 & \text{ if } \ x,y\in \h, \\
{}[x,y]_{\h} & \text{ if } \ x \in\h, y\in \rr, \\
 [x,y] & \text{ if } \ x, y\in \rr.
\end{cases}
\eeq
Then $\{\,\,,\,\}_0+\{\,\,,\,\}_\infty$ is the initial Lie--Poisson bracket in $\gS(\q)$. Hence these brackets are 
compatible. In this setting, the \PC\ subalgebra of $\gS(\q)$ given by the Lenard--Magri is denoted by 
$\gZ_{\lg\h,\rr\rg}$. Here all brackets $\{\,\,,\,\}_t=\{\,\,,\,\}_0+t\{\,\,,\,\}_\infty$ with $t\in \bbk\setminus\{0\}$ 
are isomorphic~\cite{bn}. Hence for the family $\{\,\,,\,\}_t$ ($t\in\BP$), we have 
$\BP_{\sf reg}\supset \bbk\setminus\{0\}$. If $\ind\q_{(0)},\ind\q_{(\infty)}>\ind\q$, then
$\BP_{\sf sing}=\{0,\infty\}$ and $\gZ_{\lg\h,\rr\rg}$ is generated by $\cz_t$ with $t\ne 0,\infty$.
It can be shown in general that if $\BP_{\sf sing}\ne\varnothing$, then $\trdeg\gZ_{\lg\h,\rr\rg}<\bb(\q)$. 
For this reason, we are primarily interested in splittings with $\BP_{\sf reg}=\BP$. By~\cite{T}
(cf. also~\cite[Section\,2]{bn}), for a reductive $\q$,
this property is equivalent to that both $\h$ and $\rr$ are spherical subalgebras of $\q$.

Let $\gZ_\times$ denote the subalgebra generated by the Poisson centres $\cz_t$ $(t\ne 0,\infty)$. By the previous paragraph, we have $\gZ_\times\subset \gZ_{\lg\h,\rr\rg}$ regardless of 
the properties of $\q_{(0)}$ and $\q_{(\infty)}$.
The following result is implicit in \cite[Thm.\,3.2]{bn}. It demonstrates the utility of bi-homogeneous 
components of invariants in the setting of splittings.

\begin{prop}    \label{prop:Z-x}  \,
The algebra $\gZ_\times\subset \gS(\q)$ is equal to the algebra generated by the bi-homogeneous 
components of all homogeneous $F\in\gS(\q)^\q$. In particular, if $F_1,\dots, F_m$ is a generating set for 
$\gS(\q)^\q$, then $\gZ_\times$ is generated by the bi-homogeneous components of $F_1,\dots,F_m$.
\end{prop}
\begin{proof}
The argument with Vandermonde determinant used in \cite{bn} for a reductive Lie algebra $\g$ goes 
through without changes for arbitrary Lie algebras. 
\end{proof}
\begin{cl}
The bi-homogeneous components of all symmetric invariants of $\q$ commute \wrt\ any bracket in
the family $\{\ ,\,\}_t$ ($t\in\BP$). 
\end{cl}

\begin{rmk}      \label{rem:raznye-dopoln}
It can happen that, for a given $\h\subset\q$, there are different complementary subalgebras $\rr$ and 
$\rr'$. Then the bi-homogeneous components of $F\in\gS(\q)$ and contractions $\rr\ltimes\h^{\sf ab}$, 
$\rr'\ltimes\h^{\sf ab}$ can be essentially different. This provides different \PC\ subalgebras 
$\gZ_{\lg\h,\rr\rg}$ and $\gZ_{\lg\h,\rr'\rg}$. For instance, if $\h=\be$ in a reductive $\g$, then one can 
take $\ut_-$ or $\g_0$ as complementary subalgebras to $\be$ (cf. Introduction).
\end{rmk}

\subsection{On the numbers $s_0$ and $s_\infty$} 
\label{sub:more}
From now on, we deal with splittings of a reductive Lie algebra $\g$. Let $\g=\h\oplus\rr$ be an arbitrary 
splitting, i.e., $\h$ and $\rr$ are not necessarily spherical. Recall that $\h=\Lie H$ and $\rr=\Lie R$. 
For  $\g^*$, we get the induced splitting $\g^*=\h^*\oplus\rr^*$, where $\h^*=\Ann(\rr)$ 
and $\rr^*=\Ann(\h)$. Therefore, if $\g_{(0)}$ is considered, then $\rr^*\subset\g_{(0)}^*$ is regarded as 
$H$-module $(\g/\h)^*$. Likewise, $\h^*\subset \g_{(\infty)}^*$ is regarded as $R$-module $(\g/\rr)^*$.

Set $s_0:=\trdeg \bbk((\g/\h)^*)^H$ and $s_\infty:=\trdeg \bbk((\g/\rr)^*)^R$. 

\begin{thm}              \label{sum-r}
{\sf (i)} For any splitting $\g=\h\oplus\rr$, we have $s_0+s_\infty\ge \ell$. 
\\[.2ex]
{\sf (ii)} If $s_0+s_\infty = \ell$, then 
\begin{itemize}
\item the splitting in question is non-degenerate (i.e.,  $\h$ and $\rr$ are spherical); 
\item $\rr^*\cap\g^*_{\sf reg}\ne\varnothing$ \ and \ $\h^*\cap\g^*_{\sf reg}\ne\varnothing$;
\item $s_0=\dim\rr-\dim\ut$ \ and \ $s_\infty=\dim\h-\dim\ut$.
\end{itemize}
\end{thm}
\begin{proof}
{\sf (i)}  Let $\pi_0$ be the Poisson tensor of $\g_{(0)}^*$. For $\xi\in\rr^*\in\g_{(0)}^*$, consider the 
bilinear 2-form $\pi_0(\xi)$  on $\g_{(0)}$. By definition,   
\[  
       \pi_0(\xi)(x,y)=\xi([x,y]_0) \ \text{ for } \ x,y\in\g_{(0)}.
\]
Since $[\h,\h]_{0}\subset\h$ and $[\rr,\rr]_{0}=0$, the bilinear form $\pi_0(\xi)$ vanishes on 
$\h\times\h$ and $\rr\times\rr$. Hence it provides a pairing between $\h$ and $\rr$.
For a basis of $\g$ adapted to the sum $\g=\h\oplus\rr$, the matrix of $\pi_0(\xi)$ has a block 
structure $\pi_0(\xi)=\begin{pmatrix} 0 & A \\ -A^t & 0 \end{pmatrix}$, where the rectangular matrix
$A$ corresponds to the pairing of $\h$ and $\rr$. Here $\rk A=\dim H{\cdot}\xi$, hence
$\rk\pi_0(\xi)=2\dim H{\cdot}\xi$. By Rosenlicht's theorem~\cite[\S\,2.3]{VP}, if $\xi$ is generic, then 
$
    s_0=\dim\rr-\dim H\xi=\dim\rr-\frac{1}{2}\rk\pi_0(\xi)$.
Likewise, using $\g_{(\infty)}^*$ and the Poisson tensor $\pi_\infty$, 
we obtain $s_\infty=\dim\h-\frac{1}{2}\rk\pi_\infty(\eta)$ for a generic point $\eta\in\h^*$.  

Let $\pi$ be the Poisson tensor of $\g^*$. Since $\g_{(0)}$ and $\g_{(\infty)}$ are contractions of $\g$,
we have $\rk\pi_0\le \rk\pi$ and $\rk\pi_\infty\le \rk\pi$. Therefore, for generic points $\xi$ and $\eta$, we
obtain
\begin{multline}    \label{eq:s-0+s-inf}
   s_0+s_\infty=\dim\g-\frac{1}{2}(\rk\pi_0(\xi)+\rk\pi_\infty(\eta)) \ge \\ 
\dim\g -\frac{1}{2}(\rk\pi_0+\rk\pi_\infty)\ge\dim\g - \rk\pi=\ell .
\end{multline}

{\sf (ii)} Suppose that $s_0+s_\infty=\ell$.
By Eq.~\eqref{eq:s-0+s-inf},    
this equality  is equivalent to that 
$\rk\pi_0(\xi)=\rk\pi_\infty(\eta)=\dim\g-\ell$. In particular, $\rk\pi_0=\rk\pi_\infty=\rk\pi$, i.e.,
$\ind\g_{(0)}=\ind\g_{(\infty)}=\ind\g$ and the subalgebras $\h$ and $\rr$ are spherical, 
see~\cite[Sect.\,2]{bn}.  

Now, consider $\xi\in \rr^*$ as element of $\g^*$. For the same adapted basis of $\g$ as above, the 
matrix of $\pi(\xi)$ is  $\begin{pmatrix} 0 & A \\ -A^t & C \end{pmatrix}$, where 
$C$ is the matrix of $\pi(\xi)|_{\rr\times\rr}$. Note that if $x\in\h$, $y\in\rr$, and $\xi\in\rr^*$, then
$\xi([x,y]_0)=\xi([x,y])$. Hence the submatrix $A$ is the same in $\pi_0(\xi)$ and $\pi(\xi)$, i.e.,
$\pi_0(\xi)$ is obtained from $\pi(\xi)$ by replacing $C$ with zero. Therefore, 
$\rk\pi(\xi)\ge 2\rk A= \rk\pi_0(\xi)$. This implies that $\rk\pi(\xi)=\rk\pi_0(\xi)=\dim\g-\ell$ and hence
$\xi\in\g^*_{\sf reg}$. Similarly, $\eta\in\g^*_{\sf reg}$. 

Since $s_0=\dim\rr-\frac{1}{2}\rk\pi_0(\xi)$ and $\rk\pi_0(\xi)=\dim\g-\ell=2\dim\ut$ for 
$\xi\in\rr^*\cap\g^*_{\sf reg}$, the last assertion for $s_0$ follows (and likewise for $s_\infty$).
\end{proof}

\subsection{More results on $s_0$ and $s_\infty$}    
\label{subs:pro-r-and-s}
For any homogeneous space $G/H$, the linear $H$-action on $\Ann(\h)=\h^\perp\subset\g^*$ is called
the {\it isotropy representation} of $H$. Let $c(G/H)$ and $r(G/H)$ denote the {\it complexity\/} and 
{\it rank\/} of the $G$-variety $G/H$, respectively. We refer to \cite{p99} for generalities on $c(\cdot)$
and $r(\cdot)$. It is known that
\[
      \trdeg\bbk(\h^\perp)^H=2c(G/H)+r(G/H) ,
\]
see~\cite{kn90, p90, T}. In particular, if $H$ is spherical, then $\trdeg\bbk(\h^\perp)^H=r(G/H)$. There is a
dense open subset $\Omega_H\subset\h^\perp$ such that the function
\[ 
         (x\in\Omega_H) \mapsto (\dim\h^x,\ind\h^x)\subset \BN\times\BN
\] 
is constant. The assertion on $\dim\h^x$ is a standard invariant-theoretic result (see 
e.g.~\cite[\S\,1.4]{VP}, whereas the fact on $\ind\h^x$ is proved in \cite{R}. Let us write $\h_\star$ for any stabiliser $\h^y$ with $y\in\Omega_H$. Then $\dim\h_\star=\min_{x\in\h^\perp}\dim\h^x$.
Using Ra\"is' formula for the index of semi-direct products~\cite{R} and
known results on the complexity and rank, we derive more properties for $s_0+s_\infty$ in the setting of
splittings of $\g$. As before, $\g=\h\oplus\rr$, $\g_{(0)}=\h\ltimes\rr^{\sf ab}$, and
$\g_{(\infty)}=\rr\ltimes\h^{\sf ab}$. Since $\h^\perp\simeq (\g/\h)^*$ and $\rr^\perp=(\g/\rr)^*$, we have
\beq       \label{eq:s-c-r}
      s_0=2c(G/H)+r(G/H)  \ \ \& \ \  s_\infty=2c(G/R)+r(G/R) .
\eeq

\begin{thm}   \label{thm:main-r&c}
Let $\g=\h\oplus\rr$ be an arbitrary splitting. Then
\begin{itemize}
\item[\sf (i)] \ $\dim\h_\star+\dim\rr_\star=s_0+s_\infty\ge\rk\g$;
\item[\sf (ii)] \  $\ind\h_\star=\rk\g-r(G/H)$ \ and \ $\ind\rr_\star=\rk\g-r(G/R)$;
\item[\sf (iii)] \  $\bb(\h_\star)+\bb(\rr_\star)=\rk\g+c(G/H)+c(G/R)$.
\end{itemize}
\begin{proof}
{\sf (i)} By Rosenlicht's theorem~\cite[\S\,2.3]{VP}, 
\[
  s_0=\dim\h^\perp-\dim\h+\dim\h_\star \ \text{ and } \ s_\infty=\dim\rr^\perp-\dim\rr+\dim\rr_\star .
\]
By Theorem~\ref{sum-r}{\sf (i)}, taking the sum yields the assertion.
 
{\sf (ii)} By Ra\"is' formula~\cite{R}, one has
$\ind\g_{(0)}=\trdeg\bbk(\h^\perp)^H+\ind\h_\star=s_0+\ind\h_\star$. On the other hand, $\ind\g_{(0)}=\rk\g+2c(G/H)$~\cite{bn,T}. Comparing the two expressions for $\ind\g_{(0)}$ and using 
Eq.~\eqref{eq:s-c-r}, we obtain $\ind\h_\star=\rk\g-r(G/H)$. The argument for $\ind\rr_\star$ is the same.

{\sf (iii)} Recall that $\bb(\q)=\frac{1}{2}(\dim\q+\ind\q)$. Combining {\sf (i)} and {\sf (ii)}, we get
\[
       \bb(\h_\star)+\bb(\rr_\star)=\rk\g+\frac{1}{2}(s_0-r(G/H))+\frac{1}{2}(s_0-r(G/R)) ,
\]
which is exactly what is required.  
\end{proof}
\end{thm}
\begin{cl}   \label{cl:nod-degener}
For a non-degenerate splitting of $\g$, the following conditions are equivalent:
\begin{enumerate}
\item $s_0+s_\infty=\rk\g$;
\item $r(G/H)+r(G/R)=\rk\g$;
\item $\h_\star$ and\/ $\rr_\star$ are abelian Lie algebras.
\end{enumerate}
\end{cl}
\begin{proof}
In the non-degenerate case, one has $\bb(\h_\star)+\bb(\rr_\star)=\rk\g$. Since 
$\dim\h_\star+\dim\rr_\star\ge\rk\g$, we obtain $\ind\h_\star+\ind\rr_\star\le \rk\g$. It remains to observe that $\ind\h_\star\le \dim\h_\star$ and the equality exactly means that $\h_\star$ is abelian.
\end{proof}

\begin{rmk}      \label{rem:conj-equal}
It is likely that $s_0+s_\infty=\ell$ for all non-degenerate splittings of $\g$. At least, this is true for all 
examples in~\cite{bn} and the examples considered in Sections~\ref{sect:horo}-\ref{sect:ohne} below. We
can also prove this, if both $G/H$ and $G/R$ are quasiaffine.
\end{rmk}

\subsection{The non-degenerate case}    
\label{subs:sph}
Suppose now that both subalgebras $\h$ and $\rr$ are spherical. Hence 
$\ind\g_{(0)}=\ind\g_{(\infty)}=\ell$ and $\gZ_{\lg\h,\rr\rg}$ is the algebra generated by all centres $\cz_t$, 
$t\in\BP$. The order of summands in the sum $\g=\h\oplus\rr$ is assumed to be fixed. That is, if 
$F\in\gS(\g)$ is homogeneous, then $F_\bullet$ (resp. $F^\bullet$) is the bi-homogeneous component of 
$F$ of maximal degree with respect to the first (resp. second) summand.

\begin{thm}[{\cite[Sect.\,3]{bn}}]            \label{thmZ}
For a non-degenerate splitting, $\gZ_{\lg\h,\rr\rg}$ is generated by $\cz_0$, $\cz_\infty$, and 
all bi-homogeneous components of $F_1,\dots,F_\ell$, i.e.,
\beq   \label{eq:bihom}
   \{(F_j)_{i,d_j-i} \mid  j=1,\dots,\ell \ \& \ i=0,1,\dots,d_j\}.  
\eeq
Moreover, $\trdeg\gZ_{\lg\h,\rr\rg}=\bb(\g)$. 
\end{thm}
Although $\gZ_\times$ 
can be a proper subalgebra of $\gZ_{\lg\h,\rr\rg}$, it appears to be sufficiently large.

\begin{lm}          \label{lm:trdeg}
One has $\trdeg\gZ_\times=\bb(\g)$.
\end{lm}
\begin{proof}
As every $F\in\gS(\g)^\g$ is a polynomial in $F_1,\dots,F_\ell$, any bi-homogeneous component 
$F_{s,t}$ of $F$ is a polynomial in the generators \eqref{eq:bihom},  
i.e., $F_{s,t}\in \gZ_\times$. In particular,  $\cz\gS(\g)^\bullet\subset\gZ_\times$. We also have
$\cz\gS(\g)^\bullet\subset\cz_0$ and $\ind\g_{(0)}=\ell$.
Using Lemma~\ref{bullet}, we then obtain
\[
    \ell=\trdeg\cz\gS(\g)^\bullet\le \trdeg\cz_0\le\ind\g_{(0)}=\ell .
\]
Hence the extension $\cz\gS(\g)^\bullet\subset\cz_0$ is algebraic and each element of  $\cz_0$ is 
algebraic over the field of fractions of $\gZ_\times$. Changing the roles of $\h$ and $\rr$, we arrive at the 
same conclusion for $\cz_\infty$.
Thus, $\trdeg\gZ_\times=\trdeg\gZ_{\lg\h,\rr\rg}=\bb(\g)$. 
\end{proof}

In general, {\sf g.g.s.} for $\h$ and $\rr$ can be quite different (if there are any). But, in the happy occasion that $\h$ and $\rr$ share the same {\sf g.g.s.}, we get several nice properties.

\begin{prop}           \label{prop:ggs2}
Suppose that $F_1,\ldots,F_\ell$ is a {\sf g.g.s.} for both $\h$ and $\rr$. Then $\gZ_\times$ is a polynomial ring that is freely generated by the bi-homogeneous components $(F_j)_{i,d_j-i}$ ($1\le j\le \ell$) such that $d_j-\deg_{\h} F_j^\bullet \le i \le \deg_{\rr}(F_j)_\bullet$. The total number of such components equals $\bb(\g)$.
\end{prop}
\begin{proof}
Since $F_1,\ldots,F_\ell$ is a {\sf g.g.s.} for both $\h$ and $\rr$, using Theorem~\ref{thm:kot14}{\sf (ii)}
with $\me=\rr$ or $\h$, we obtain
\[
   \sum_{j=1}^\ell \deg_{\rr} F_j^\bullet=\dim\rr \quad \& \quad \sum_{j=1}^\ell \deg_{\h}(F_j)_\bullet = \dim\h .
\]  
For any $j$, the $\rr$-degrees of the nonzero bi-homogeneous components of $F_j$ belong to the interval
$[\deg_{\rr} (F_j)_\bullet, \deg_{\rr} F_j^\bullet]$. Hence there are at most 
$\deg_{\rr} F_j^\bullet - \deg_{\rr} (F_j)_\bullet + 1$ of them.
Since $\deg_{\rr} (F_j)_\bullet = d_j - \deg_{\h}  (F_j)_\bullet$, 
 the total number of non-zero bi-homogeneous components $\{(F_j)_{i,d_j-i}\}$ does not exceed 
\[
    \sum_{j=1}^\ell (\deg_{\rr} F_j^\bullet+\deg_{\h}(F_j)_\bullet-d_j+1)=
\dim\rr+\dim\h- \bb(\g)+\ell
=\bb(\g).
\] 
Since $\trdeg \gZ_\times=\bb(\g)$ (Lemma~\ref{lm:trdeg}), all these components are nonzero and 
algebraically independent. 
\end{proof}

\begin{thm}            \label{thm:ggs2}
Suppose that both algebras $\cz_0$ and $\cz_\infty$ have algebraically independent generators
and $F_1,\ldots,F_\ell$ is a common {\sf g.g.s.} for $\h$ and $\rr$. Then $\gZ_{\lg\h,\rr\rg}$  is 
a polynomial ring. 
\end{thm}
\begin{proof}
By Proposition~\ref{prop:ggs2}, the algebra $\gZ_\times$ is freely generated by the 
set 
\[
    \eus M=\{(F_j)_{i,d_j-i} \mid 1\le j\le \ell,\, d_j-\deg_{\h} F_j^\bullet \le i \le \deg_{\rr}(F_j)_\bullet \}. 
\]
and $\#\eus M=\bb(\g)$.  
Now, we may replace $F_1^\bullet,\ldots,F_\ell^\bullet\in\cz_0\cap\eus M$ with a Hilbert basis of $\cz_0$
and $(F_1)_\bullet,\ldots,(F_\ell)_\bullet\in\cz_\infty\cap\eus M$ with a Hilbert basis of $\cz_\infty$.
By Theorem~\ref{thmZ}, the resulting set of polynomials generates $\gZ_{\lg\h,\rr\rg}$, and it
still contains  $\bb(\g)$ elements.
\end{proof}

\begin{rmk}
{\sf (i)} \ If there is no common {\sf g.g.s.} for $\h$ and $\rr$, then above arguments do not apply. Still, in 
some cases $\gZ_{\lg\h,\rr\rg}$ can be a polynomial ring, see Theorems~\ref{thm:sum} and \ref{thm:bb}. 
\\  \indent
{\sf (ii)} \ In Section~\ref{sec:g+t}, we come across a curious case in which $\cz_0$ and 
$\cz_\infty$ are polynomial rings and there does exist a common \ggs\ for $\h$ and $\rr$. 
\end{rmk} 

Recall that $H$ and $R$ are connected subgroups of $G$ (they are spherical now),
$s_0:=\trdeg \bbk((\g/\h)^*)^H$, and $s_\infty:=\trdeg \bbk((\g/\rr)^*)^R$. Below we provide a simpler 
proof for the inequality $s_0+s_\infty\ge\ell$ in the non-degenerate case and derive more properties of
$\gZ_{\lg\h,\rr\rg}$. 

\begin{prop}       \label{sumr}
We have $s_0+s_\infty\ge \ell$. If $s_0+s_\infty= \ell$, then all the components $(F_j)_{i,d_j-i}$ with 
$1\le j\le \ell$ and $0< i < d_j$ are nonzero and algebraically independent. 
\end{prop}
\begin{proof} 
The number of all components $(F_j)_{i,d_j-i}$ with $1\le j\le \ell$ and $0< i < d_j$ is equal to 
$\sum_{j=1}^\ell(d_j-1)=\bb(g)-\ell$. ({\sl A priori}, some of them can be equal to $0$.) Note that the 
extreme values $i=0$ and $i=d_j$ are excluded here.

If $i=0$ and $(F_j)_{0,d_j}\ne 0$, then 
$(F_j)_{0,d_j}=F_j^\bullet\in\gS(\rr^{\rm ab})^\h$, i.e.,  $(F_j)_{0,d_j}\in\cz_0$. 
Under the natural identifications $\rr^*\simeq\Ann(\h)\simeq (\g/\h)^*$, we have 
$\gS(\rr)^\h \simeq\bbk[(\g/\h)^*]^H$. Therefore, at most $s_0$ elements among 
$\{(F_j)_{0,d_j}\mid j=1,\dots,\ell\}$ are algebraically independent. The similar argument applies to 
$s_\infty$, $i=d_j$, and $(F_j)_{d_j,0}$.  Then taking all bi-homogeneous components shows that 
$\trdeg \gZ_\times \le \bb(\g)-\ell+s_0+s_\infty$. On the other hand,
by Lemma~\ref{lm:trdeg}, one has $\trdeg \gZ_\times=\bb(\g)$. Hence $s_0+s_\infty\ge \ell$. 

If $s_0+s_\infty=\ell$, then the previous paragraph shows that there are at most $\ell$ algebraically 
independent elements among $\{(F_j)_{0,d_j}, (F_j)_{d_j,0}\mid 1\le j\le\ell\}$. Hence all the components 
with $0< i < d_j$ must be nonzero and algebraically independent. 
\end{proof}

The following useful result is similar to Theorem~\ref{thm:ggs2}. The difference is that  instead of the presence of a common 
\ggs\ we require that $s_0+s_\infty= \ell$.
\begin{thm}            \label{thm:sum}
Suppose that both algebras $\cz_0$ and $\cz_\infty$ have algebraically independent generators and 
$s_0+s_\infty= \ell$. Then $\gZ_{\lg\h,\rr\rg}$ is a polynomial ring that is freely generated by the set 
$\widetilde{\eus M}$ consisting of  $\{(F_j)_{i,d_j-i} \mid 1\le j\le \ell,\, 0< i < d_j\}$, a Hilbert basis of 
$\cz_0$, and a Hilbert basis of $\cz_\infty$.
\end{thm}
\begin{proof}
The nonzero bi-homogeneous components of the form $(F_j)_{0,d_j}$ and $(F_j)_{d_j,0}$ with
$1\le j\le \ell$ belong to $\cz_0$ and $\cz_\infty$, respectively. Therefore, the algebra generated by 
$\widetilde{\eus M}$ contains all the  centres $\cz_t$, $t\in\BP$, i.e., it coincides with $\gZ_{\lg\h,\rr\rg}$. 
Furthermore,
\[
  \# \widetilde{\eus M}=  \bb(\g)-\ell+s_0+s_\infty=\bb(\g)=\trdeg \gZ_{\lg\h,\rr\rg}.
\]
Thus, the elements of $\widetilde{\eus M}$ are algebraically independent and $\gZ_{\lg\h,\rr\rg}$ is a 
polynomial ring. 
\end{proof}
\begin{rmk}
Formally, the presence of \ggs\ for $\h$ or $\rr$ is not required in Theorem~\ref{thm:sum}.
Nevertheless, we often need a \ggs\ for $\h$ (resp. $\rr$) in order to prove that $\cz_0$ (resp.
$\cz_\infty$) is a polynomial ring, cf. Theorem~\ref{thm:inv-h}. But, in this setting, \ggs\ for $\h$ and $\rr$
can be different.
\end{rmk}
 
\section{Horospherical contractions and splittings}
\label{sect:horo}

\noindent
Let $\g=\ut_+\oplus\te\oplus\ut_-$ be a fixed triangular decomposition and 
$\be_\pm=\ut_\pm\oplus\te$. Let $(\ ,\,)$ be a $G$-invariant non-degenerate symmetric bilinear form 
on $\g$. Then the bilinear form $\lg\ ,\,\rg:=(\ ,\,)\vert_\te$ on $\te$ is also non-degenerate. 

Let $\h_+$ be a horospherical subalgebra of $\g$ that lies in $\be_+$, i.e., 
$\ut_+\subset\h_+\subset\be_+$. Then $\h_+=\ut_+\oplus\te_1$, where $\te_1=\h_+\cap\te$. Note that 
$\h_+$ is a spherical subalgebra of $\g$ and $\lg\ ,\,\rg$ is non-degenerate on $\te_1$. 
Let $\te_0\subset\te$ be an algebraic subalgebra such that $\te_0\oplus\te_1=\te$. For instance, one 
can take $\te_0=\te_1^\perp\subset \te$, the orthogonal complement {w.r.t{.}} $\lg\ ,\,\rg$. Set 
$\h_-=\ut_-\oplus\te_0$. Then $\g=\h_+\oplus\h_-$ is a non-degenerate splitting, which we 
call a {\it horospherical splitting}. Since the summands here play symmetric 
roles, we concentrate on $\q:=\h_+\ltimes \h_-^{\sf ab}$ and \ggs\ for $\h_+$. We also say that  $\q$ is a 
{\it horospherical contraction\/} of $\g$. Recall that $F^\bullet$ stands for the bi-homogeneous component 
of $F\in \gS(\g)$ with minimal $\h_+$-degree ({=}\,maximal $\h_-$-degree). 

\begin{thm}               \label{thm:ggs-h}
Let $F_1,\dots, F_\ell$ be an arbitrary Hilbert basis of $\gS(\g)^\g$.  Set $A=\{j\mid F^\bullet_j\in \gS(\te_0)\}$
and $a=\#A$. Then 
\begin{itemize}
\item[\sf (1)] \ $\dim\te_0\le a$;
\item[\sf (2)] \ $F_1,\dots, F_\ell$ is a\/ {\sf g.g.s.} for $\h_+$ \ if and only if \ $\dim\te_0=a$;
\item[\sf (3)] \  if the equality holds in {\sf (1)}, then $F_j^\bullet\in\h_+\otimes\gS(\h_-)$ for all $F_j$ such
that $F_j^\bullet\not\in\gS(\te_0)$. That is, the bi-degree of such $F_j^\bullet$ is $(1,d_j-1)$ and
$\deg_{\h_-}\!(F_j)=d_j-1$.
\end{itemize}
\end{thm}
\begin{proof}  {\sf (1)} \ Since $\te$ normalises both $\h_+$ and $\h_-$,
the bi-homogeneous components of any $F\in \gS(\g)^\g$ are $\te$-stable. Therefore, if 
$F_i^\bullet\in\gS(\h_-)$, then $F_i^\bullet\in \gS(\te_0)$. We then obtain
\[
   F_i^\bullet \in \gS(\te_0) \Longleftrightarrow F_i^\bullet\in\gS(\h_-) \Longleftrightarrow 
   \deg_{\h_-}\!(F_i)=\deg F_i .
\]
On the other hand, if $F_i^\bullet\not\in\gS(\h_-)$, then clearly $\deg_{\h_-}\!(F_i)\le \deg F_i-1$. Hence 
\beq    \label{eq:a}
  \sum_{i=1}^{\ell}\deg_{\h_-}\!(F_i) \le \sum_{i=1}^{\ell} \deg F_i - (\ell -a) = \dim\be-\ell+a=\dim\ut+ a. 
\eeq       
By Theorem~\ref{thm:kot14}{\sf (i)}, one has 
$\dim\ut_-+ \dim\te_0=\dim\h_-\le \sum_{i=1}^{\ell}\deg_{\h_-}\!(F_i)$, i.e., $\dim\te_0\le a$.
\\ \indent
{\sf (2)} \ If $F_1,\dots, F_\ell$ is a \ggs, then $\{F^\bullet_j\mid j\in A\}$ are algebraically independent.
Therefore, $a=\#A\le \dim\te_0$ and hence $a=\dim\te_0$. The converse follows from 
Theorem~\ref{thm:kot14}{\sf (ii)} and \eqref{eq:a}.
\\ \indent
{\sf (3)} \ 
If $a=\dim\te_0$, then one must have the equality in~\eqref{eq:a}, which implies that
$\deg_{\h_-}\!(F_j)=\deg F_j -1$ for each $F_j$ such that  
$F_j^\bullet\not\in\gS(\te_0)$. 
\end{proof}

\begin{rmk}   
\label{rem:generated-by}
If $F_1,\dots, F_\ell$ is a\/ \ggs\ for $\h_+$, then $F_1^\bullet,\dots,F_l^\bullet$ are algebraically 
independent and $\trdeg\cz\gS(\q)=\ell$. 
However, this does not necessarily mean that $\cz\gS(\q)$ is generated by $\{F_i^\bullet\}$.
The definition of $\q$ shows that $\te_0$ is the centre of $\q$, hence
$\te_0\subset\cz\gS(\q)$. Therefore, the polynomials $\{F_i^\bullet\}$ can generate $\cz\gS(\q)$
only if a Hilbert basis of $\gS(\g)^\g$ contains $\dim\te_0$ elements of degree~1, i.e.,
$\dim\te_0=\dim\z(\g)$, where $\z(\g)$ is the centre of $\g$. In particular, if $\g$ is semisimple, then the 
equality $\cz\gS(\q)=\bbk[F_1^\bullet,\dots,F_\ell^\bullet]$ can only hold if $\h_+=\be$ and
$\h_-=\ut_-$. (And by~\cite[Theorem\,3.3]{alafe} this is really the case.)
\end{rmk}

\subsection{}
Our next goal is to demonstrate that $\q=\h_+\ltimes \h_-^{\sf ab}$ has a polynomial ring $\cz\gS(\q)$ 
whenever $\h_+$ admits a \ggs\ To prove that certain elements of $\cz\gS(\q)=\bbk[\q^*]^Q$ generate the whole ring of invariants, we use 
the Igusa criterion~\cite[Lemma~4]{Ig}. Other proofs of this result can be found in~\cite[Thm.\,4.12]{VP} or~\cite[Lemma~6.1]{rims}.

\begin{lm}[Igusa]              \label{lm-Ig}
Let an algebraic group $L$ act regularly on an irreducible affine variety $X$ and let $C$ be an integrally 
closed finitely generated subalgebra of\/ $\bbk [X]^L$. Suppose that the morphism 
$\psi: X\to \spe C =: Y$ has the properties:
\begin{itemize}
\item[{\sf (i)}] $\Ima\psi$ contains an open subset $\Omega\subset Y$ such that $\codim(Y\setminus\Omega) \ge 2$;
\item[{\sf (ii)}] there is a  dense open subset  $\Xi\subset Y$ such that $\psi^{-1}(y)$ contains a dense $L$-orbit for all $y\in \Xi$.
\end{itemize}
Then $C = \bbk [X]^L$. \emph{[In particular, the algebra of $L$-invariants is finitely generated.]}  
\end{lm}

As before, if $\gt v$ is a subspace of $\te$, then $\gt v^\perp\subset\te$ is its orthogonal complement  
w.r.t. $\lg\ ,\,\rg$. We say that $x\in\gt v$ is generic, if $\g^x=\g^{\gt v}$. 

\begin{lm}               \label{lm:t}
Let $\te_1\subset \te$ be a subspace and $x\in\te_1^\perp$ a generic point. Letting $\es:=[\g^x,\g^x]$, 
we have $\te\cap \es\subset \te_1$ and $\te\cap\es$ is a Cartan subalgebra of\/ $\es$.  
\end{lm}
\begin{proof}
Write $\g^x=\z\oplus\es$, where $\z$ is the centre of $\g^x$. Then $(\z,\es)=0$. Since $x\in \te_1^\perp$ 
is generic, we have $\te_1^\perp\subset \z$. Hence $\z^\perp = \es\cap\te$ is contained in $\te_1$. 
\end{proof}

Let $Q$ be a connected algebraic group with $\Lie Q=\q$. It is the semi-direct product of 
the connected subgroup $H_+\subset G$ and the commutative unipotent group $\exp(\h_-^{\sf ab})$.
Since $\h_+$ is spherical, we have $\ind\q=\ell$ and  $\trdeg \bbk(\q^*)^Q=\ell$. 

\begin{lm}             \label{lm:dense}
{\sf (i)} \ For any $x\in\te_1^\perp$, the affine space $x+\ut_+$ is $H_+$-stable. 
Moreover, if $x\in\te_1^\perp$ is generic, then $H_+$ has a dense orbit in $x+\ut_+$; 
\\ 
{\sf (ii)} \  one has $\max_\xi (\dim H_+{\cdot}\xi)=\dim\ut$, where $\xi$ ranges over 
$\te_1^\perp\oplus\ut$.
\end{lm}
\begin{proof}
{\sf (i)} Since $H_+{\cdot}x\subset x+\ut_+$ and $H_+{\cdot}\ut_+\subset\ut_+$, the subvariety 
$x+\ut_+\subset\g$ is $H_+$-stable. For $x\in\te_1^\perp$, the Levi subalgebra $\g^x$ inherits the 
triangular decomposition from $\g$.  
Let $e\in \g^x\cap \ut_+$ be a regular nilpotent 
element of $\es=[\g^x,\g^x]$. Then $x+e\in\g_{\sf reg}$ and $\g^{x+e}=\g^x\cap\g^e=\z\oplus\es^e$. Therefore
$\h_+^{x+e}=(\te_1\cap\z)\oplus \es^e$ and  
$\dim\es^e=\rk\es$. By Lemma~\ref{lm:t}, if $x$ is generic, then $\dim(\te_1\cap\z)=\dim\te_1-\rk\es$. 
Hence $\dim\h_+^{x+e}= \dim\te_1$ and $\dim H_+{\cdot}(x+e)=\dim\ut_+$.

{\sf (ii)} Let $\Omega\subset\te_1^\perp$ be the set of generic elements. Then $\Omega\times\ut$ is a 
dense open subset of $\h_+^\perp=\te_1^\perp\oplus\ut$ and $\dim H_+{\cdot}\xi=\dim \ut$ for all
$\xi\in \Omega\times\ut$.
\end{proof}

\begin{thm}               \label{thm:inv-h}
If there is a \ggs\ for $\h_+$, then $\cz\gS(\q)$ is a polynomial ring. More precisely, if
$F_1,\ldots,F_{\ell}$ is a \ggs, then $\cz\gS(\q)$ is freely generated by a basis of\/ $\te_0$ and the 
polynomials $\{F_j^\bullet\}$ that are not contained in $\gS(\te_0)$. 
\end{thm}
\begin{proof}
We are going to apply Lemma~\ref{lm-Ig} to $X=\q^*$, $L=Q$, and the algebra $C\subset\gS(\q)^Q$ 
generated by a basis of $\te_0$ and the highest components $\{F_i^\bullet\}$ that are not contained in 
$\gS(\te_0)$. Then we get the morphism $\psi: \q^*\to \spe C$.

\textbullet\quad Since $F_1,\ldots,F_{\ell}$ is a {\sf g.g.s.} for $\h_+$, the generators of $C$ are 
algebraically independent. In particular, $C$ is integrally closed and $\spe C\simeq\BA^\ell$ is an affine space. 
Set $a=\dim\te_0$. By Theorem~\ref{thm:ggs-h}, we may assume that $F_i^\bullet \in \gS(\te_0)$ if 
$ i\le a$ and  $F_i^\bullet \in \h_+\otimes\gS(\h_-)$ if $i\ge a+1 $. 

\textbullet\quad Let us show that $\psi$ is onto. Write $\q^*=\h_+^*\oplus\h_-^*$,  where $\h_-^*$ is 
identified with $\Ann(\h_+)\simeq\h_+^\perp=\ut\oplus\te_1^\perp$. 
The elements of $C$ that form a basis of $\te_0$ provide a surjective map from $\h_-^*$ to $\BA^{a}\simeq \te_0^*$. 

Take an arbitrary $x\in\te_1^\perp\subset\h_-^*$  and consider the affine subspace 
$x+\ut\subset\h_-^*$. There is $y \in x+\ut$  that is a regular element of $\g\simeq\g^*$. 

For a bi-homogeneous component $F_{j,d-j}$ of $F\in\gS^d(\g)$, we have $\textsl{d}_y F_{j,d-j}=0$, if 
$j\ge 2$. Suppose that $i>a$. Then 
$F_i^\bullet\in \h_+\otimes\gS(\h_-)$ and $\eta:=\textsl{d}_y F_i=\textsl{d}_y F_i^\bullet$. Furthermore, 
$\eta\in\h_+$ and as a linear function on $\h_+^*$  it coincides with the 
restriction of $F_i^\bullet$ to $y+\h_+^*\subset\q^*$. By a result of Kostant~\cite[Theorem~9]{ko63}, we have 
$\langle \textsl{d}_y F_j \mid 1\le j\le\ell\rangle_{\bbk}=\g^y$. In particular, the restrictions
${F_i^\bullet}|_{y+\h_+^*}$ with $a{+}1\le i \le \ell$ are linearly independent. Hence they define a 
 surjective map 
from $y+\h_+^*$ to $\BA^{\ell-a}$. Thus, $\psi$ is indeed surjective. 

\textbullet\quad Our next task is to study generic fibres of $\psi$.  
For any $x\in\te_1^\perp$, the affine subspace $x+\ut\subset\h_-^*$ is $H_+$-stable. 
Suppose that $x\in\te_1^\perp$ is generic, i.e., $x\in\Omega$.  
By Lemma~\ref{lm:dense},
$H_+$ has a dense orbit  $ H_+{\cdot}(x+e)\subset x+\ut$ with $e\in\gt u$.  
Note that $e$ does not depend on $x$.  

Since $C$ contains a basis of $\te_0$ and $\te_1^\perp\simeq\te_0^*$, we have $\psi(x+\ut+\h_+^*)\cap\psi(\tilde x+\ut+\h_+^*)=\varnothing$, whenever $\tilde x\in\te_1^\perp$ is not 
equal to $x$. There is a dense open subset 
$V\subset \spe C$ such that
for each $v\in V$, the fibre $\psi^{-1}(v)$ contains a point $\xi=u+t+e$, where $u\in\h_+^*$ and  $t\in\Omega$, 
i.e.,  
here $H_+{\cdot}(t+e)$ is dense in $t+\ut$. Set $y=t+e$.  Then 
$\psi^{-1}(v) \cap (\h_+^* \times H_+{\cdot}y)$ is a non-empty open subset of $\psi^{-1}(v)$ for each $v\in V$. 

Let us say that an irreducible  component $\Psi_v$ of $\psi^{-1}(v)$ with $v\in V$ is 
 not {\it suitable}, if $\Psi_v\cap X_0=\varnothing$ for $X_0:=\h_+^* \times H_+{\cdot}y$. If $\psi^{-1}(v)$ has a 
non-suitable component, then $v$ is  {\it non-suitable} as well. 
Let $D_{\sf ns}\subset\q^*$ be the union of all non-suitable components.  We have 
\[
D_{\sf ns} \subset \h_+^* \times \left (\bigcup_{x\in \Omega } (x+\ut) \setminus H_+{\cdot}\tilde y \right) =
\h_+^* \times \left ( (\Omega\times \gt u) \setminus H_+{\cdot}(\Omega\times \bbk e)\right),
\]
where $\tilde y=x+e$ and $H_+{\cdot}\tilde y$ is dense in $x+\ut$ for each $x$. 
Note that the dimension of each 
irreducible component $\Psi_v$ is at least $\dim\g-\ell$. 
Since $\dim \overline{D_{\sf ns}}\le \dim\g -1$ and a generic fibre of 
$\psi\!: \overline{D_{\sf ns}}\to \overline{\psi(D_{\sf ns})}$ has dimension at least  $\dim\g - \ell$, the image $\psi(D_{\sf ns})$ is contained in a proper closed subset of $\spe C$. Therefore we can 
replace $V$ with a  non-empty open subset consisting of suitable points. Now 
$\psi^{-1}(v)\cap X_0$  is dense in $\psi^{-1}(v)$ for each $v\in V$. 

Let us prove that  
$\overline{Q{\cdot}\xi}=\psi^{-1}(v)$.  This assertion will follow from the fact that the $Q$-orbits on $X_0$
 are separated by the invariants $F_i^\bullet$ with $i>a$.  Below, we describe 
these orbits and prove the assertion. 

Clearly, $\h_+^*$ and $\h_-^*$ are $H_+$-stable. Further,  for $\tilde u+z\in\q^*$ with 
$\tilde u\in\h_+^*$ and $z\in\h_-^+$, we have 
\[
H_-{\cdot}(\tilde u+z)=\exp(\h_-){\cdot}(\tilde u+z)=\tilde u+z+\ad^*(\h_-)(\tilde u+z)=\tilde u+z+\ad^*(\h_-)z
\]
with $\ad^*(\h_-)z\subset\h_+^*$, since $[\h_-,\h_-]_0=0$ and $[\h_-,\h_+]_0\subset \h_-$. 
By definition, $\ad^*(\h_-)z$ is zero on $\h_+^z$. Then by dimension reasons $\ad^*(\h_-)z=\Ann(\h_+^z)\subset\h^*$. 

We have to show that each $\xi_2\in\psi^{-1}(v)\cap X_0$ lies in $Q{\cdot}\xi$. It is safe to assume that 
$\xi_2=u_2+y$ with $u_2\in\h_+^*$ and $y=t+e$ as above. 
Then 
$\xi_2\in Q{\cdot}\xi$ if and only if  
\[
u_2\in H_+^y{\cdot}u + \ad^*(\h_-)y=\{ g{\cdot}u+\ad^*(h)y \mid g\in H_+^y,\, h\in\h_-\}=H_+^y{\cdot}u+\Ann(\h_+^y).
\]  
This is the case if and only if the restrictions of $u$ and $u_2$ to $\h_+^y$ lie in one and the same orbit of $H_+^y$. 
 
The  isotropy group $G^{y}$ is abelian and $H_+^y\subset G^y$. Hence $H_+^y$ acts on $(\h_+^y)^*$ 
trivially. Thus,  $\xi_2\in Q{\cdot}\xi$ if and only if the restrictions ${u}|_{\h_+^{y}}$ and ${u_2}|_{\h_+^{y}}$ 
coincide. Since $F_i^\bullet|_{y+\h_+^{y}}$ with $i>a$ form a  basis of $\h_+^{t+e}$, the equalities  
$F_i^\bullet(\xi)=F_i^\bullet(\xi_2)$ with $i > \dim a$ imply that ${u}|_{\h_+^{y}}={u_2}|_{\h_+^{y}}$. 
Thereby $\xi_2\in Q{\cdot}\xi$. 

Thus, all conditions of Lemma~\ref{lm-Ig} are satisfied and hence $\bbk[\q^*]^Q=C$. 
\end{proof}

\begin{rmk}        \label{extrem}
The cases with $\te_1=\te$ or $\{0\}$ in Theorem~\ref{thm:inv-h} has been treated before. 
If $\h_+=\be$, then $\cz\gS(\gt b\ltimes\ut_-^{\rm ab})=\bbk[F_1^\bullet,\ldots,F_\ell^\bullet]$
for {\bf any} Hilbert basis $(F_1,\dots, F_\ell)$, see~\cite[Theorem\,3.3]{alafe}.
If $\h_+=\ut$, then $\te_0=\te$ and $\cz\gS(\ut\ltimes(\g/\gt b)^{\rm ab})=\gS(\te)$, see~\cite[Prop.\,4.2]{bn}.
\end{rmk}

Recall from Section~\ref{sect:PC} that $s_0=\trdeg \bbk((\g/\h_+)^*)^{H_+}=
\trdeg\bbk(\h_+^\perp)^{H_+}$.
\begin{prop}     \label{prop-s}
For a solvable horospherical subalgebra $\h_+=\ut\oplus\te_1$, we have 
$s_0=\ell-\dim\te_1$.
\end{prop}
\begin{proof}
Recall that the $H_+$-module $(\g/\h_+)^*$ is isomorphic to $\ut\oplus\te_1^\perp\subset\g$. By 
Rosenlicht's theorem~\cite[\S\,2.3]{VP}, $s_0=\dim(\g/\h_+)-\max_\xi(\dim H_+{\cdot}\xi)$ for 
$\xi\in \ut\oplus\te_1^\perp$. By Lemma~\ref{lm:dense}{\sf (ii)}, we have 
$\max_\xi(\dim H_+{\cdot}\xi)=\dim\ut$. Hence 
$s_0=\dim\g-\dim\h_+-\dim\ut=\ell-\dim\te_1$. 
\end{proof}

\begin{cl}
For any horospherical splitting $\g=\h_+\oplus\h_-$, we have $s_0+s_\infty=\ell$.
\end{cl}

\subsection{A necessary condition for the presence of \ggs} 
\label{subs:necessary}
Let us translate Theorem~\ref{thm:ggs-h} in the setting of functions on $\g$. Using $(\ ,\,)$, we identify
$\gS(\g)$ with $\bbk[\g]$. In particular,  we consider elements of $\gS(\g)^\g$ as $G$-invariant
polynomials on $\g$. The celebrated theorem of Chevalley asserts that the restriction homomorphism
$\bbk[\g]\to \bbk[\te]$, $F\mapsto F\vert_\te$, induces an isomorphism $\bbk[\g]^G\simeq\bbk[\te]^W$, 
where $W$ is the {\it Weyl group\/} of $(\g,\te)$. 
\\ \indent
{\bf Convention.} In the rest of this section, we assume that $\te_0=\te_1^\perp$, i.e., 
$\te_0\oplus\te_1=\te$ is the orthogonal direct sum. Then the
condition that $F^\bullet_j\in \gS(\te_0)$ for $F\in\gS(\g)^G$ is equivalent to that
$F_j\vert_{\te_0}\not\equiv 0$, if $F$ is regarded as element of $\bbk[\g]^G$. 
\\ \indent
Hence parts {\sf (i),(ii)} of Theorem~\ref{thm:ggs-h} can be restated as follows.

\begin{prop}    \label{restate1}
A Hilbert basis $F_1,\dots,F_\ell$ of\/ $\bbk[\g]^G$ is a \ggs\ for $\h_+=\ut_+\oplus\te_1$ 
if and only if there are exactly $\dim\te_0$ indices $j\in\{1,\dots,\ell\}$ such that 
$F_j\vert_{\te_0}\ne 0$.
\end{prop}

Using the Chevalley isomorphism above, one can work inside $\bbk[\te]^W$ and give a version of 
Proposition~\ref{restate1} without ``indices''. Clearly, the composition 
\[  
    \te_1^\perp=\te_0\hookrightarrow \te\stackrel{\tau}{\longrightarrow} \te/W=\spe\!(\bbk[\te]^W) \simeq \BA^\ell  
\] 
is a finite morphism. Therefore,
\begin{itemize}
\item \ $\tau(\te_0)$ is a closed subvariety of $\te/W$ and $\dim\tau(\te_0)=\dim\te_0$;
\item \ ${\sf res}_0(\bbk[\te]^W):= \bbk[\te]^W\vert_{\te_0}$ is a $\bbk$-algebra of Krull dimension 
$\dim\te_0$;
\item \ $\tau(\te_0)=\spe\!({\sf res}_0(\bbk[\te]^W)$.
\end{itemize}

\begin{prop}    \label{restate2}
Suppose that $\h_+=\te_1\oplus\ut_+\subset\be_+$ 
and\/ $\te_0=\te_1^\perp\subset\te$. 
Then
\\[.6ex]
\centerline{
there is a \ggs\ for\/ $\h_+$ \ $\Longleftrightarrow$ \ ${\sf res}_0(\bbk[\te]^W)$ is a graded polynomial ring.}
\end{prop}
\begin{proof}
``$\Leftarrow$". Here $p:\bbk[\te]^W \to {\sf res}_0(\bbk[\te]^W)$ is a surjective homomorphism 
of degree zero of graded polynomial rings.  By~\cite[Lemma\,4.1]{Ri}, there is a Hilbert basis 
$F_1,\dots,F_\ell$ of $\bbk[\te]^W$ such that $p(F_1),\dots,p(F_a)$ is a Hilbert basis of 
${\sf res}_0(\bbk[\te]^W)$ and  $p(F_j)=0$ for $j>a$. This yields a required \ggs

 ``$\Rightarrow$". This readily follows from Proposition~\ref{restate1}.
\end{proof}

Using a result of Richardson, we give a practical necessary condition for the existence of \ggs\ for a 
solvable horospherical subalgebra $\h_+$. By Proposition~\ref{restate2},
if $\h_+$ admits a \ggs, then ${\sf res}_0(\bbk[\te]^W)$ is a polynomial ring (in $\dim\te_0$ variables),
i.e., $\tau(\te_0)$ is an affine space. In particular, $\tau(\te_0)$ is normal. Set 
\[
N_W(\te_0)=\{w\in W\mid w(\te_0)\subset\te_0\} \ \& \ 
Z_W(\te_0)=\{w\in W\mid w(t)=t \text{ for all } t\in\te_0\} . 
\]
The finite group $W_0=N_W(\te_0)/Z_W(\te_0)$ acts on $\te_0$ and, clearly, 
${\sf res}_0(\bbk[\te]^W)\subset \bbk[\te_0]^{W_0}$. Since $\te_0$ is normal, the following is a special case
of Theorem~A in~\cite[\S\,3]{Ri}.
\begin{prop}[R.W.\,Richardson]   \label{prop:rwr} 
In the notation above, one has
 \\[.5ex]
\centerline{ $\tau(\te_0)$ is normal \ $\Longleftrightarrow$ \ ${\sf res}_0(\bbk[\te]^W)=\bbk[\te_0]^{W_0}$.}
\end{prop}
A consequence that is relevant for us is
\begin{cl}     \label{cor:suffice}
If\/ ${\sf res}_0(\bbk[\te]^W)\subsetneqq \bbk[\te_0]^{W_0}$, then there is no \ggs\ for $\h_+$.
\end{cl}

\section{A  horospherical splitting of $\g\dotplus\te$.} 
\label{sec:g+t}

\noindent 
As before, $\g=\ut_+\oplus\te\oplus\ut_-$ is a fixed triangular decomposition of  $\g$ and
$\be_\pm=\te\oplus\ut_\pm$. Consider the reductive Lie algebra $\tilde\g=\g\dotplus\te$, where the second summand is another copy of the Cartan subalgebra of $\g$. Then 
$\bb(\tilde\g)=\bb(\g)+\ell$ and $\rk\tilde\g=2\ell$. Here $\te\dotplus\te\subset\tilde\g$ splits into the sum of 
$\te_0=\{(t,t) \mid t\in\te\}$ and $\te_1:=\{(t,-t)\mid t\in\te\}$. We equip $\tilde\g$ with the invariant 
bilinear form $(\ ,\,)\dotplus\lg\ ,\,\rg$. Then $\te_0$ and $\te_1$ are orthogonal w.r.t. this form.
To distinguish the two copies of $\te$, the direct summand in $\tilde\g$ is denoted below by 
$\te_{\sf dir}$.

Set $\tth=\ut_+\oplus\te_1$ and $\trr=\ut_-\oplus\te_0$. Then both subalgebras are spherical in $\tilde\g$ 
and isomorphic to a Borel subalgebra of $\g$, and also $\tth\cap\trr=\{0\}$. Hence $\tilde\g=\tth\oplus\trr$ 
is a horospherical splitting. Following our general scheme, we set 
$\tilde\g_{(0)}=\tth\ltimes{\trr}^{\rm ab}$ and $\tilde\g_{(\infty)}=\trr\ltimes{\tth}^{\rm ab}$. Then 
$\tilde\cz_0$ and $\tilde\cz_\infty$ are the corresponding Poisson centres. 

\begin{rmk}   \label{rem:Drinfeld}
It is easily seen that if $\be$ is a Borel subalgebra of $\g$, then $\tth\simeq\trr\simeq \be$ and 
\[
    \tilde\g_{(0)}\simeq \tilde\g_{(\infty)}\simeq\be\ltimes (\g/\ut)^{\sf ab}= \be\ltimes(\be^*)^{\sf ab}. 
\]
In~\cite{BR}, the Lie algebra $\mathcal I(\be):=\be\ltimes(\be^*)^{\sf ab}$ is called the {\it Drinfeld double} 
of  $\be$. Our interpretation of $\mathcal I(\be)$ as a contraction of $\g\dotplus\te_{\sf dir}$ is new. This can be generalised to the arbitrary parabolic subalgebras of $\g$.
We hope to elaborate on this topic in a forthcoming publication.
\end{rmk}
Let $F_1,\dots,F_\ell$ be a Hilbert basis of $\gS(\g)^\g$ and  $\xi_1,\ldots,\xi_\ell$ a 
basis for  $\te_{\sf dir}$. They form together a Hilbert basis of $\gS(\tilde\g)^{\tilde\g}$. In order to obtain 
a \ggs\ for $\tth$ or $\trr$, certain modifications of this Hilbert basis is needed.
Here $\xi_i^\bullet$ and $(\xi_i)_\bullet$ with $1\le i\le \ell$ form bases of $\te_0$ and $\te_1$, 
respectively. As before, we identify $\te$ with $\te^*\subset\g^*$. For any $F\in\gS(\g)^\g$, set 
$f=F\vert_{\te}$ and let $\fb$ be a copy of $f$ in $\gS(\te_{\sf dir})$.  
\begin{prop}      
\label{lm:dva-ggs}   
Using the notation above, we have
\begin{itemize}
\item[\sf (i)] \ $\xi_1,\ldots,\xi_\ell, F_1 - \fb_1, \ldots, F_\ell- \fb_\ell$ is a\/ {\sf g.g.s.} for  $\tth$;
\item[\sf (ii)] \ $\xi_1,\ldots,\xi_\ell, F_1 - (-1)^{d_1} \fb_1, \ldots, F_\ell-(-1)^{d_\ell} \fb_\ell$ is a\/ {\sf g.g.s.} for $\trr$.
\end{itemize}
\end{prop}
\begin{proof}
Let $\{h_i \mid 1\le i\le\ell\}$ be a basis of $\te\subset\g$. We assume that the identification of 
$\te$ and $\te_{\sf dir}$ is given by $h_i\mapsto\xi_i$. Then $\{h_i+\xi_i\mid 1\le i\le\ell\}$ is a basis for
$\te_0$ and $\{h_i-\xi_i\mid 1\le i\le\ell\}$ is a basis for $\te_1$. This yields the bi-homogeneous 
decompositions 
\beq    \label{eq:bi-hom-basis}
    h_i=\frac{1}{2}\bigl((h_i+\xi_i)+(h_i-\xi_i)\bigr) \quad \& \quad  
    \xi_i=\frac{1}{2}\bigl((h_i+\xi_i)+(-h_i+\xi_i)\bigr). 
\eeq
For every $F_j$, the restriction $f_j=F_j\vert_{\te}\in \gS(\te)$ is a non-zero $W$-invariant polynomial in  
$h_1,\dots,h_\ell$. Then $\fb_j\in\gS(\te_{\sf dir})$ is obtained from $f_j$ by replacing each $h_i$ with 
$\xi_i$. 

{\sf (i)} It  follows from~\eqref{eq:bi-hom-basis} that the bi-homogeneous component with maximal
$\trr$-degree, $\fb_j^\bullet$, is obtained from $\fb_j$ by replacing each $\xi_i$ with its $\trr$-component, 
i.e., $\frac{1}{2}(h_i+\xi_i)$. Likewise, the bi-homogeneous component $F_j^\bullet$ is obtained from 
$f_j$ by replacing each $h_i$ with its $\trr$-component, i.e., $\frac{1}{2}(h_i+\xi_i)$. Consequently, 
$F_j^\bullet$ and $\fb_j^\bullet$ belong to $\gS(\te_0)\subset\gS(\trr)$, and they are equal. Taking 
$F_j{-}\fb_j$  cancels out this highest component, which implies that 
$(F_j{-}\fb_j)^\bullet \not\in\gS(\te_0)$ and hence $\deg_{\trr}(F_j{-}\fb_j)^\bullet\le d_j-1$ for 
$j=1,\dots,\ell$. Since $\dim\te_0=\ell$, part {\sf (i)} follows  from Theorem~\ref{thm:ggs-h}.

{\sf (ii)} A similar idea works for $(F_j)_\bullet$ and $(\fb_j)_\bullet$, i.e., the bi-homogeneous 
components with maximal $\tth$-degree. To obtain $(\fb_j)_\bullet$, we replace each $\xi_i$ in $\fb_j$ 
with its $\tth$-component, i.e., $\frac{1}{2}(-h_i+\xi_i)$. Similarly, to obtain $(F_j)_\bullet$, we replace 
each $h_i$ in $f_j$ with its $\tth$-component, i.e., $\frac{1}{2}(h_i-\xi_i)$. Hence $F_j^\bullet$ 
and $\fb_j^\bullet$ belong to $\gS(\te_1)\subset\gS(\tth)$, and they are equal up to a sign.
Taking $F_j{-}(-1)^{d_j}\fb_j$ cancels out these components and therefore
$\deg_{\tth}(F_j{-}(-1)^{d_j}\fb_j)^\bullet\le d_j-1$ for $j=1,\dots,\ell$.
Since $\dim\te_1=\ell$, part {\sf (i)} follows from Theorem~\ref{thm:ggs-h}.
\end{proof}
Since there are \ggs\ for $\tth$ and $\trr$, Theorem~\ref{thm:inv-h} implies that the Poisson
centres $\tilde\cz_0$ and $\tilde\cz_\infty$ are polynomial rings, i.e., we obtain
\begin{cl}                    \label{cor:double}
For the Drinfeld double $I(\be)$, the Poisson centre $\cz\gS(\mathcal I(\be))$ is a polynomial ring.
\end{cl}
Note that here $\tilde\cz_0$ and $\tilde\cz_\infty$ are generated by the related highest components, since 
$\{\xi_i^\bullet \mid 1\le i\le\ell\}$ and  $\{(\xi_i)_\bullet\mid 1\le i\le\ell\}$ are bases of $\te_0$ and $\te_1$, 
respectively. This agrees with Remark~\ref{rem:generated-by}, because here 
$\dim\te_0{=}\dim\te_1{=}\dim\z(\tilde\g)=\ell$. It is worth mentioning, however, that $I(\be)$ does not 
always have the {\sl codim}--$2$ property. For, the image of $\te$ in $\g/\ut$ is the centre of $I(\be)$ and 
hence $I(\be)\simeq \bigl(\be\ltimes (\g/\be)^{\sf ab}\bigr)\dotplus \te$. By~\cite[Theorem\,4.2]{alafe}, the  
Lie algebra $\be\ltimes\ut_-^{\sf ab}$ has the {\sl codim}--$2$ property only for $\g=\sln$.

If all the degrees $\{d_j\}$ are even, then Proposition~\ref{lm:dva-ggs} provides a common \ggs\ for 
$\tth$ and $\trr$. Then Theorem~\ref{thm:ggs2} already guarantees us that $\gZ_{\lg\tth,\trr\rg}$ is a 
polynomial ring. (If $\g$ is simple, then odd degrees occur only for $\g$ of type $\eus{A}_{n}$ ($n\ge 2$), 
$\eus{D}_{2n+1}$, and $\eus{E}_{6}$.) Nevertheless, the polynomiality of $\gZ_{\lg\tth,\trr\rg}$ follows in 
general from other results of Section~\ref{sect:PC}. Our main result on the horospherical splitting 
$\tilde\g=\tth\oplus\trr$ can be summarised as follows.

\begin{thm}          \label{thm:bb}
The\/ \PC\ subalgebra $\gZ_{\lg\tth,\trr\rg}\subset\gS(\g\dotplus\te)$ is a polynomial ring. It is freely 
generated by  $\{\xi_i^\bullet, (\xi_i)_\bullet\mid 1\le i\le\ell\}$, and the bi-homogeneous components 
$\{(F_i)_{j,d_i-j}\mid 1\le i\le \ell,\ 0< j<d_i\}$. 
\end{thm}
\begin{proof}
By Proposition~\ref{prop-s}, we have $s_0=s_\infty=\ell$. Therefore $s_0+s_\infty=2\ell=\rk\tilde\g$.
Since there are \ggs\ for both $\tth$ and $\trr$ (Proposition~\ref{lm:dva-ggs}), the algebras 
$\tilde\cz_0$ and $\tilde\cz_\infty$ are polynomial rings, see Theorem~\ref{thm:inv-h}. Thus, 
Theorem~\ref{thm:sum} applies here and it provides the desired description of free generators for  
$\gZ_{\lg\tth,\trr\rg}$. 
\end{proof}

\section{Semi-horospherical splittings and involutions} 
\label{sect:ohne}

\noindent
A splitting $\g=\h\oplus\rr$ is said to be {\it semi-horospherical}, if either $\h$ or $\rr$ is solvable and 
horospherical. Let $\sigma$ be an involution of a simple Lie algebra $\g$ and $\g=\g_0\oplus\g_1$ the 
corresponding $\BZ_2$-grading. Write $G_0$ for the connected subgroup of $G$ with $\Lie G_0=\g_0$.
We say that $\sigma$ is $S$-{\it regular}, if $\g_1$ contains regular semisimple elements of $\g$. Our 
first goal is to attach a semi-horospherical splitting to an $S$-regular involution $\sigma$.
This possibility is briefly discussed in \cite[Remark\,6.3]{bn}. 
Then we look at the related \PC\ subalgebras. 
We refer to~\cite{kr71} for basic notions and invariant-theoretic results related to $\BZ_2$-gradings.

Let $\ce\subset\g_1$ be a Cartan subspace. Since $\ce$ contains regular semisimple elements, 
$\te:=\g^\ce$ is a Cartan subalgebra of $\g$ and $\te=\te_0\dotplus\ce$, where $\te_0:=\g^{\ce}\cap\g_0$.
Note that $\lg\ce,\te_0\rg=0$. As $\sigma(\te)=\te$, the involution $\sigma$ acts on the
root system $\Delta=\Delta(\g,\te)$ and hence on the standard $\BQ$-form $\te_\BQ\subset\te$. 
Then $\te_{\BQ}=(\te_0)_{\BQ}\oplus\ce_{\BQ}$, where $\ce_\BQ=\ce\cap\te_\BQ$. 
Since $\g^{\ce}=\te$, no root $\ap\in\Delta$ vanishes on $\ce$. Choose a generic $\xi\in\ce_{\BQ}$ 
such that $\ap(\xi)\ne 0$ for all $\ap\in\Delta$ and set  $\Delta^+:=\{\ap\in\Delta \mid \ap(\xi)>0\}$. Let 
$\be$ be the Borel subalgebra of $\g$ associated with $\Delta^+$ and $\ut=[\be,\be]$. 

\begin{lm}      \label{lm:b}
We have\/ $\be+\g_0=\g$ and\/ $\be\cap\g_0=\te_0$.
\end{lm} 
\begin{proof}
Since $\sigma(\xi)=-\xi$, we have $\sigma(\ap)\in\Delta^-$ for each $\ap\in\Delta^+$. Then 
$\be\cap\g_0\subset \te\cap\g_0$ and hence $\be\cap\g_0=\te_0$. For each $x\in\ut_-=[\be_-,\be_-]$, 
we have  $\sigma(x)\in\ut$ and $x+\sigma(x)\in\g_0$. Thus, $x=(x+\sigma(x)) - \sigma(x)\in \g_0+\be$. 
\end{proof}

Let $\Pi=\{\ap_1,\dots,\ap_\ell\}$ be the set of simple roots in $\Delta^+$. Since 
$\sigma(\Delta^+)=-\Delta^+$, we have $\sigma(\Pi)=-\Pi$. This allows us to describe the {\it Satake 
diagram\/} of $\sigma$, $\mathsf{St}(\sigma)$, for the $S$-regular involutions. Here $\mathsf{St}(\sigma)$ 
is just the Dynkin diagram of $\g$ equipped with extra arrows connecting certain pairs of nodes. Namely,
if $\sigma(\ap_i)=-\ap_j$ and $i\ne j$, then we draw an arrow connecting $\ap_i$ and $\ap_j$. If
$\sigma(\ap_i)=-\ap_i$, then no arrow is attached to $\ap_i$. It follows that the arrow connecting $\ap_i$ and $\ap_j$ gives rise to the element $\ap_i-\ap_j\in\te_0\subset \te\simeq\te^*$ and $\dim\te_0$ equals
the number of arrows in $\mathsf{St}(\sigma)$.

\begin{rmk}
In the setting of real semisimple Lie algebras, a general description of Satake diagrams can be found in
\cite[Ch.\,4, \S\,4.3]{t41}. Then a bijection between the real forms and involutions of 
$\g$~\cite[Ch.\,4, \S\,1.3]{t41} allows to attach the Satake diagram to any involution. A straightforward 
construction of $\mathsf{St}(\sigma)$ via $\sigma$ is given in~\cite[1.1-1.4]{DeC-P} (although the term
``Satake diagram'' is not used there).
\end{rmk}
Set $\h=\ce\oplus\ut$. Then $\h\cap\g_0=\{0\}$ and $\g=\h\oplus\g_0$ is a non-degenerate 
semi-horospherical splitting of $\g$. If $\sigma=\vartheta^{\sf max}$, then $\ce=\te$, $\h=\be$, and 
the polynomial ring $\gZ_{\lg\be,\g_0\rg}$ is described in~\cite[Sect.\,4]{bn}. Below, we consider the 
remaining $S$-regular involutions $\sigma$, i.e., those with $\te_0\ne \{0\}$. The list of related pairs 
$(\g,\g_0)$ is given in the introduction. It is convenient, however, to distinguish the parity of $n$ in the 
$\sln$-case.

\begin{table}[h]
\begin{tabular}{c|llllc|}
   & $\g$ & $\g_0$ & $\dim\ce$ & $\dim\te_0$ &  \\ \hline
1 & $\mathfrak{sl}_{2n}$ & $\sln\dotplus\sln\dotplus\bbk$ & $n$ & $n{-}1$ & $n\ge 2$ \\
2 & $\mathfrak{sl}_{2n+1}$ & $\sln\dotplus\mathfrak{sl}_{n+1}\dotplus\bbk$ & $n$ & $n$ & $n\ge 1$
 \\
3 & $\mathfrak{so}_{2n}$ & $\mathfrak{so}_{n+1}\dotplus \mathfrak{so}_{n-1}$ & $n{-}1$ & $1$ & $n\ge 4$  \\
4 & $\eus E_6$ & $\mathfrak{sl}_6\dotplus\tri$ & $4$ & $2$ & \\ \hline
\end{tabular} 
\vskip1ex
\caption{The $S$-regular involutions of simple $\g$ that are not of maximal rank}  \label{tableau}
\end{table}
\noindent
We wish to apply Theorem~\ref{thm:sum} to semi-horospherical splittings $\g=\h\oplus\g_0$ 
related to $\sigma$ in Table~\ref{tableau}. Here $\g_{(0)}=\h\ltimes\g_0^{\sf ab}$ and 
$\g_{(\infty)}=\g_0\ltimes\h^{\sf ab}$. Our hope is to prove that $s_0+s_\infty=\ell$ \ and both 
$\cz_0=\cz\gS(\g_{(0)})$ and $\cz_\infty=\cz\gS(\g_{(\infty)})$ are polynomial rings.

\textbullet \ \ The first part of this program is
\begin{lm}      \label{lm:5-items}
For all semi-horospherical splittings related to $\sigma$ of 
Table~\ref{tableau}, we have $s_0=\dim\te_0$, $s_\infty=\dim\ce$, and hence $s_0+s_\infty=\ell$.
\end{lm}
\begin{proof} 
1. Since $\g_1$ is an orthogonal $G_0$-module, we have $(\g/\g_0)^*\simeq \g_1$ and 
$\bbk(\g_1)^{G_0}$ is the fraction field of $\bbk[\g_1]^{G_0}$. A basic property of the isotropy representation $G_0\to SO(\g_1)$  is that $\bbk[\g_1]^{G_0}$ is a polynomial
ring in $\dim\ce$ variables~\cite{kr71}. Hence $s_\infty=\dim\ce$.

2. Since $\h$ is solvable and horospherical, Proposition~\ref{prop-s} shows that $s_0=\ell-\dim\ce$.
\end{proof}

\textbullet \ \ In~\cite[Sect.\,5]{coadj}, it is shown that $\cz_\infty$ is a polynomial ring for all items in 
Table~\ref{tableau}. 

\textbullet \ \ By Theorem~\ref{thm:inv-h}, if $\h$ admits a \ggs, then $\cz_0$ is a polynomial ring. Hence 
our task is to realise whether there is a \ggs\ for $\h=\ce\oplus\ut$ associated with Items 1--4 in 
Table~\ref{tableau}. Here we can forget about $\g_0$ and deal with the complementary to $\h$ 
horospherical subalgebra $\h_-=\te_0\oplus\ut_-$, where $\te=\te_0\oplus\ce$.

As in Proposition~\ref{restate1}, we regard the elements of $\gS(\g)^\g$ as polynomial functions on 
$\g^*\simeq\g$. Recall that the condition that $F^\bullet\in \gS(\te_0)$ is being transformed into the 
property that $F\vert_{\te_0}\ne 0$. This means that, for a fixed Cartan 
subalgebra $\te\subset\g$, one has to know the embedding $\te_0\subset\te$ and then describe the 
restriction of $\bbk[\te]^W$ to $\te_0$, see Proposition~\ref{restate2}. As explained above, the explicit embedding $\te_0\subset\te$ is 
obtained from $\mathsf{St}(\sigma)$.

\subsection{$\g=\mathfrak{sl}_{2n}$}
\label{subs:gl1}
For $\g_0=\sln\dotplus\sln\dotplus\bbk$, the Satake diagram is presented in Figure~\ref{S_A1}. We assume that $n\ge 2$, so that the diagram does have arrows and $\te_0\ne \{0\}$.

\begin{figure}[htb]
\begin{picture}(115,38)(20,5)
\setlength{\unitlength}{0.024in}
\multiput(10,8)(90,0){2}{\circle{5}}
\multiput(40,8)(15,0){3}{\circle{5}}
\put(8.65,6.5){\tiny 1}  \put(53.65,6.6){\tiny $n$}
\multiput(43,8)(15,0){2}{\line(1,0){9}}
\multiput(13,8)(60,0){2}{\line(1,0){4}}
\multiput(33,8)(60,0){2}{\line(1,0){4}}
\multiput(19.5,5)(60,0){2}{$\cdots$}   
{\color{blue}
\multiput(19.5,13)(60,0){2}{$\cdots$} 
\put(55,14){\oval(90,25)[t]}
\put(55,14){\oval(30,18)[t] }
\multiput(10,14)(90,0){2}{\vector(0,-1){3}}
\multiput(40,14)(30,0){2}{\vector(0,-1){3}}
}
\put(66,3){$\underbrace{\mbox{\hspace{38\unitlength}}}_{n-1}$}
\put(6,3){$\underbrace{\mbox{\hspace{38\unitlength}}}_{n-1}$}
\put(110,12.8){\vector(-2,-1){10}}  \put(111,12.8){\tiny $2n-1$}
\end{picture}
\caption{The Satake diagram for $(\mathfrak{sl}_{2n}, \sln\dotplus\sln\dotplus\bbk)$} 
\label{S_A1}
\end{figure}

\noindent
The arrows connect the pair of nodes $(1,2n-1),(2,2n-2), \dots, (n-1,n+1)$. 
Therefore $\te_0$ is the $\bbk$-linear span of $\ap_i-\ap_{2n-i}$ with $1\le i\le n-1$. Hence
\beq          \label{eq:c1}
   \te_0=\{{\sf diag}(c_1,\dots,c_n,c_n,\dots, c_1)\mid \sum_{i=1}^nc_i=0\} .
\eeq
Let ${\tt P}_2,\dots,{\tt P}_{2n}$ be a Hilbert basis of $\bbk[\mathfrak{sl}_{2n}]^{SL_{2n}}$ 
such that ${\tt P}_k(A)=\tr(A^k)$ for $A\in\g\simeq\g^*$. Then ${\bt}_k:={\tt P}_k\vert_{\te}$ is a power 
sum, i.e., $\bt_k(A)$ is the sum of the $k$-th powers of the diagonal entries of $A$.  Here
$W=\BS_{2n}$, $N_W(\te_0)/Z_W(\te_0)\simeq \BS_n$, and
$\BS_n$ acts on $\te_0$ by permuting the entries $\{c_i\}$. It follows from~\eqref{eq:c1} that 
${\tt P}_k|_{\te_0}\in \bbk[\te_0]^{\BS_n}$ for every $k$ and the polynomials ${\tt P}_k|_{\te_0}$ with 
$2\le k\le n$ generate the ring $\bbk[\te_0]^{\BS_n}$. Therefore, if $k>n$, then there is a polynomial 
$\mathcal F$ in $n-1$ variables such that $\widetilde{\tt P}_k=\ttP_k-\mathcal F(\ttP_2,\ldots,\ttP_n)$ 
vanishes on $\te_0$.  By Theorem~\ref{thm:ggs-h}, this means that 
${\tt P}_2,\dots,{\tt P}_n,\widetilde{\tt P}_{n+1},\dots,\widetilde{\tt P}_{2n}$ is a \ggs\ for $\h$. 

Thus, Theorem~\ref{thm:sum} applies here and hence $\cz_{\lg\h,\g_0\rg}$ is a polynomial ring.

\subsection{$\g=\mathfrak{sl}_{2n+1}$}
\label{subs:gl2}
For $\g_0=\sln\dotplus\slno\dotplus\bbk$, the Satake diagram is

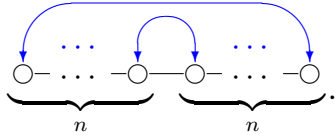
\begin{figure}[h]
\begin{picture}(120,34)(10,2)
\setlength{\unitlength}{0.020in}
\multiput(10,8)(75,0){2}{\circle{5}} \multiput(40,8)(15,0){2}{\circle{5}}
\put(43,8){\line(1,0){9}}
\multiput(13,8)(45,0){2}{\line(1,0){4}} \multiput(33,8)(45,0){2}{\line(1,0){4}}
\multiput(19.5,5)(45,0){2}{$\cdots$}   
{\color{blue}
\multiput(19.5,13)(45,0){2}{$\cdots$} 
\put(47.5,14){\oval(75,25)[t]} \put(47.5,14){\oval(15,18)[t] }
\multiput(10,14)(75,0){2}{\vector(0,-1){3}}
\multiput(40,14)(15,0){2}{\vector(0,-1){3}}
}
\put(51,3){$\underbrace{\mbox{\hspace{38\unitlength}}}_{n}$}
\put(6,3){$\underbrace{\mbox{\hspace{38\unitlength}}}_{n}$}
\end{picture}.
\end{figure}
\noindent
The arrows connect the pair of nodes $(1,2n),(2,2n-1), \dots, (n,n+1)$. Therefore here
\beq          \label{eq:c2}
   \te_0=\{{\sf diag}(c_n,\dots,c_1, c_0 ,c_1,\dots, c_n)\mid c_0+2\sum_{i=1}^n c_i=0\} .
\eeq
First, consider the case with $n=1$. Here $\te_0=\{\diag(c,-2c,c)\}\subset\mathfrak{sl}_3$ and the
$SL_3$-invariants of degree 2 and 3 are uniquely determined (up to a scalar multiple). 
For $\xi\in\te_0$, we have
\[
          \ttP_2(\xi)=6c^2, \quad {\tt P}_3(\xi)=-6c^3.
\]
Therefore, one always has $a=2> 1=\dim\te_0$ (in the notation of 
Theorem~\ref{thm:ggs-h}). Hence there is no \ggs\ for $\h\subset\mathfrak{sl}_3$.
Below, we give a general argument based on results of Section~\ref{subs:necessary}.

Using Eq.~\eqref{eq:c2}, we consider $\te_0$ as linear space of dimension $n$ with coordinates 
$c_1,\dots,c_n$. Here $W=\BS_{2n+1}$ and it readily follows 
from~\eqref{eq:c2} that $N_W(\te_0)/Z_W(\te_0)\simeq \BS_{n}$. The symmetric group $W_0=\BS_n$ 
acts on $\te_0$ by permuting the entries $c_1,\dots,c_n$, i.e., $\te_0$ is the standard permutation $\BS_n$-module.
In particular, $\bbk[\te_0]^{\BS_n}$ contains an element of degree~$1$.
Therefore, the restriction homomorphism $p: \bbk[\te]^W \to \bbk[\te_0]^{W_0}$ is not onto.
Using Corollary~\ref{cor:suffice}, we obtain that $\h$ does not admit a \ggs

\subsection{$\g=\sone$ and $\g_0=\gt{so}_{n+1}\oplus\gt{so}_{n-1}$}  
\label{subs:Dn} 
If $n=2$, then $\g\simeq\tri\dotplus\tri$ and $\sigma$ permutes these summands. This is a special case 
of symmetric pairs considered in~\cite[Sect.\,7]{bn}.  If $n=3$, then $\g\simeq\mathfrak{sl}_4$ and 
$\g_0\simeq \tri\dotplus\tri\dotplus\bbk$, which is a particular case of $\sigma$ considered in
Section~\ref{subs:gl1}. Therefore, we may assume that $n\ge 4$. Then
$\mathsf{St}(\sigma)$ is depicted in Figure~\ref{S_D}. Here $\dim\ce=n{-}1$ and $\dim\te_0=1$. 

\begin{figure}[h]
\begin{picture}(100,22)(50,5)
\setlength{\unitlength}{0.0185in} %
\put(30,8){\circle{6}}     
\multiput(50,8)(40,0){2}{\circle{6}} 
\put(28.6,6.35){\tiny 1} \put(48.6,6.35){\tiny 2}
\put(33,8){\line(1,0){14}}
\multiput(110,-2)(0,20){2}{\circle{6}}
\put(108.40,-3.35){\tiny $n$}
\put(93,10){\line(2,1){13}}
\put(93,6){\line(2,-1){13}}
\multiput(53,8)(28,0){2}{\line(1,0){6}}   
\put(64,5){$\cdots$}
{\color{blue}\put(116,8){\oval(20,20)[r]}}
{\color{blue}\put(113,18){\vector(-1,0){2}}}
{\color{blue}\put(111,-2){\vector(-1,0){2}}}
\end{picture}
\caption{The Satake diagram for $(\sone, \mathfrak{so}_{n+1}\dotplus\mathfrak{so}_{n-1})$} 
\label{S_D}
\end{figure}
\noindent
The only arrow  connects the nodes $n{-}1$ and $n$, i.e., $\te_0=\bbk(\ap_{n-1}-\ap_n)$. For the usual 
choice of simple roots with $\ap_i=\esi_i-\esi_{i+1}$ ($i<n$) and $\ap_n=\esi_{n-1}+\esi_n$,
$\te_0$ is the line generated by $\esi_n$.
Suppose that $\sone$ is the set of the skew-symmetric matrices with respect to the antidiagonal. 
Then $\ut$ (resp. $\te$) is the set of such strictly upper-triangular (resp. diagonal)  matrices and \\
\centerline{
$\te_0=\{{\sf diag}(0,\dots,0,\esi_n,-\esi_n,0,\dots,0)\}$. }

\noindent
For $A\in\sone$ and $k\le n$, let $\Delta_{2k}(A)$ be the sum of all principal $2k$-minors of $A$. 
Consider the Hilbert basis $F_1,\dots,F_n$ of $\bbk[\sone]^{SO_{2n}}$ such that 
$F_k(A)=\Delta_{2k}(A)$ 
for $k<n$ and $F_n(A)={\sf Pf}(A)$, the {\it Pfaffian\/} of $A$. 
Clearly, $\Delta_{2k}\vert_{\te_0}\ne 0$ if and only if $k=1$. Since ${\sf Pf}(A)^2=\Delta_{2n}(A)$, we
see that $\Delta_{2}\vert_{\te_0}$ is the only nonzero restriction and hence
$\Delta_2,\ldots,\Delta_{2n-2}, {\sf Pf}$ \ is a {\sf g.g.s.} for $\h$.  

Thus, Theorem~\ref{thm:sum} applies here and hence $\cz_{\lg\h,\g_0\rg}$ is a polynomial ring. Note also that here $W_0\simeq\BS_2$ and the restriction $\bbk[\te]^W\to \bbk[\te_0]^{W_0}$ is onto.

\subsection{$\g=\eus E_6$}
\label{subs:E6}

Here $\g_0=\gt{sl}_6\dotplus\gt{sl}_2$ and the Satake diagram of $\sigma$ is depicted in Figure~\ref{S_E}. 
Therefore $\dim\te_0=2$ and $\dim\ce=4$. 

\begin{figure}[h]
\begin{picture}(90,50)(40,10)
\setlength{\unitlength}{0.024in} 
\multiput(23,18)(15,0){4}{\line(1,0){9}}
\put(50,3){\circle{5}}    
\put(18.6,16.5){\tiny 1} \put(33.6,16.5){\tiny 2}  \put(48.6,16.5){\tiny 3} \put(63.6,16.5){\tiny 4}
\put(78.6,16.5){\tiny 5}   \put(48.6,1.5){\tiny 6}
\multiput(20,18)(15,0){5}{\circle{5}}
\put(50,6){\line(0,1){9}}
{\color{blue}\put(50,23){\oval(60,20)[t]}
\put(50,21){\oval(31,18)[t]}
\multiput(34.6,23)(30.8,0){2}{\vector(0,-1){2}}
\multiput(20,23)(60,0){2}{\vector(0,-1){2}}}
\end{picture}
\caption{The Satake diagram for $(\eus E_6, \gt{sl}_6 \dotplus \gt{sl}_2 )$} 
\label{S_E}
\end{figure}

\noindent
The two arrows indicate that $\sigma(\ap_1)=-\ap_5$ and $\sigma(\ap_2)=-\ap_4$,
i.e., $\te_0=\lg \ap_1-\ap_5, \ap_2-\ap_4\rg$.
As in Section~\ref{subs:gl2}, we use Richardson's result to prove that there is no \ggs\ for $\h$. To these end, we have to explicitly determine $W_0=N_W(\te_0)/Z_W(\te_0)$, where $W=W(\eus E_6)$.

\begin{lm}   \label{lm:E6}
For $\te_0=\lg \ap_1-\ap_5, \ap_2-\ap_4\rg\subset\te=\te(\eus E_6)$, we have $W_0\simeq\BS_3$ and
$W_0\to GL(\te_0)$ is the standard reflection representation of\/ $\BS_3$.  
\end{lm}
\begin{proof}
We have $\te=\te_0\oplus\ce$ and $\te_0=\ce^\perp$. Hence $N_W(\ce)=N_W(\te_0)$. Since 
$\ce\subset\g_1$ is a Cartan subspace for $\sigma$, $W_\ce:=N_W(\ce)/Z_W(\ce)$ is the {\it little Weyl 
group\/} for the symmetric variety $\eus E_6/(\eus A_5\times\eus A_1)$. Therefore, 
$W_\ce\simeq W(\eus F_4)$. As $\ce$ contains regular semisimple elements of $\te$, $Z_W(\ce)=\{1\}$ 
and hence $N_W(\ce)=N_W(\te_0)\simeq W(\eus F_4)$.  To determine the (non-trivial) group $Z_W(\te_0)$, we use the restricted root system $\Delta\vert_\ce$, which is isomorphic to 
$\Delta(\eus F_4)$.

A root $\gamma=\sum_{i=1}^6 n_i\ap_i\in \Delta(\eus E_6)$ is said to be {\it symmetric}, if
$n_1=n_5$ and $n_2=n_4$, see Fig.~\ref{S_E}. All other roots are said to be non-symmetric.
It is easily seen that $\gamma$ is symmetric if and only if $\gamma\in\ce$, and all symmetric roots form the root system of type $\eus D_4$. If $\gamma$ is symmetric, then $\bar\gamma=\gamma\vert_\ce$ is identified with $\gamma$, and these roots form the set of all {\bf long} roots in $\Delta(\eus F_4)$. If
$\gamma$ is non-symmetric, then $\bar\gamma=\frac{1}{2}(\gamma+\sigma(\gamma))$ is a {\bf short} 
root in $\Delta(\eus F_4)$.

For a symmetric $\gamma$, the reflection $r_\gamma\in W(\eus F_4)\subset W(\eus E_6)$ acts trivially on $\te_0=\ce^\perp$. Hence $W(\eus D_4)\subset Z_W(\te_0)$. Note that 
$W(\eus D_4)\triangleleft W(\eus F_4)$ and $W(\eus F_4)/W(\eus D_4)\simeq
W(\eus A_2)\simeq \BS_3$. 
The simple reflections in $W(\eus A_2)$ correspond to the short simple roots of $\Delta(\eus F_4)$.
These simple roots are $\bar{\ap_1}=\bar{\ap_5}$ and $\bar{\ap_2}=\bar{\ap_4}$, and the corresponding
reflections in $W(\eus F_4)$ are $\bar r_1=r_{\ap_1}r_{\ap_5}$ and $\bar r_2=r_{\ap_2}r_{\ap_4}$.
Computing the action of $\bar r_1$ and $\bar r_2$ on $\ap_1-\ap_5$ and $\ap_2-\ap_4$, we obtain the
reflection representation of $\BS_3$. Hence $W(\eus D_4)= Z_W(\te_0)$, and we are done.
\end{proof}

The algebra $\bbk[\te_0]^{W_0}$ contains an element of degree $3$, whereas there are no elements of 
degree $3$ in $\bbk[\te]^W$. Hence the restriction homomorphism $p: \bbk[\te]^W\to \bbk[\te_0]^{W_0}$ 
is not onto. By Corollary~\ref{cor:suffice}, we conclude that $\h$ does not admit a \ggs

{\it\bfseries Conclusion to Table}~\ref{tableau}. 
For items $1\&3$, we have proved that $\h$ admits a \ggs\ and $\gZ_{\lg\h,\g_0\rg}$ is a polynomial ring.
On the other hand, $\h$ does not admit a \ggs\ for items $2\&4$. Therefore, the structure of 
$\cz_0=\cz\gS(\h\ltimes\g_0^{\sf ab})$ is not known for them. Consequently, the structure of 
$\gZ_{\lg\h,\g_0\rg}$ is also hidden in fog.

\subsection{On good generating systems for $\g_0$}
\label{subs:ggs-review}
Let $\sigma$ be an involution of a simple Lie algebra $\g$ and $\g=\g_0\oplus\g_1$ the corresponding 
$\BZ_2$-grading. It is known that there are four cases such that $\g_0$ does not admit a 
\ggs~\cite[Remark\,4.3]{coadj}. (In these ``bad'' cases, $\g$ is of type~$\eus E_n$.) For all other 
involutions, there does exist a \ggs~\cite{coadj,contr}. It is noticed  in~\cite[Section\,6]{bn} that if 
$\sigma=\vartheta^{\sf max}$, then {\bf every} Hilbert basis of $\gS(\g)^\g$ is a \ggs{}  In the setting of 
splittings, this property is quite useful, if we need a common \ggs\ for $\g_0$ and the other 
summand.

For any simple $\g$, there is a unique (up to $G$-conjugation) involution of maximal rank, which can be 
either inner or outer. Furthermore, if $\vartheta^{\sf max}$ is outer, then
there is a unique {\sl inner\/} $S$-regular involution $\vartheta^{\sf int}$. This happens for the Lie
algebras of type $\eus A_n$ ($n\ge 2$), $\eus D_{2n+1}$  ($n\ge 2$), and $\eus E_6$.
The involutions $\vartheta^{\sf int}$ occur in Table~\ref{tableau} (where $n$ must be odd in no.\,3).

\begin{prop}    \label{prop:all-ggs}
If $\sigma=\vartheta^{\sf int}$, then any Hilbert basis of\/ $\gS(\g)^\g$ is a \ggs\ for $\g_0$.
\end{prop}
\begin{proof}
It follows from Lemma~\ref{lm:b} that $\dim\g_1=\dim\be-\dim\te_0$. Since $\sigma$ is inner, $\sigma(F)=F$ for all $F\in\gS(\g)^\g\simeq \bbk[\g]^G$. For any $x\in\g_1$, we have
\[
   F(x)=(\sigma F)(x)=F(\sigma x)=F(-x)=(-1)^{\deg F}F(x) .
\]
Therefore $F\vert_{\g_1}\equiv 0$ whenever $\deg F$ is odd. Take any Hilbert basis of 
$\bbk[\g]^G$. If $\sigma$ is inner in Table~\ref{tableau}, i.e., $\sigma=\vartheta^{\sf int}$,
then $\# \{j\mid \deg F_j \text{ is odd}\}=\dim\te_0$. (Although this is readily verified case-by-case,
an a priori argument is also available.) Also, since $(\g_0,\g_1)=0$, the condition that 
$F_j\vert_{\g_1}\equiv 0$ is equivalent to that $F_j^\bullet\not\in \gS(\g_1)$. Hence
\[
      \sum_{j=1}^\ell \deg_{\g_1}\!( F^\bullet_j) \le \sum_{j=1}^\ell \deg F_j-\dim\te_0=\dim\g_1.
\]
By Theorem~\ref{thm:kot14}, there must be the equality here and, therefore, $F_1,\dots,F_\ell$ is a 
\ggs\ for $\g_0$.
\end{proof}
\begin{rmk}
For $(\mathfrak{so}_{2n}, \mathfrak{so}_{n+1}\dotplus\mathfrak{so}_{n-1})$ with $n$ even, $\sigma$ is 
outer and Proposition~\ref{prop:all-ggs} is not true. Here all $F_j$ have even degrees, and any Hilbert 
basis contains two elements of degree $n$. The corresponding $2$-dimensional space is
$\lg \Delta_n, \mathsf{Pf}\rg$ (see notation in Section~\ref{subs:Dn}).  Any basis for this space can be 
taken as part of a Hilbert basis of $\bbk[\mathfrak{so}_{2n}]^{SO_{2n}}$. But to obtain a \ggs\ for $\g_0$, 
one must include the Pfaffian $\mathsf{Pf}$ into a Hilbert basis. It is known that if $n$ is even (i.e., 
$\sigma$ is outer), then $\sigma(\mathsf{Pf})=-\mathsf{Pf}$, while $\sigma(\Delta_{2i})=\Delta_{2i}$ for 
$i\le n-1$. This shows that $\mathsf{Pf}\vert_{\g_1}\equiv 0$ regardless of the parity of $n$.
\end{rmk}

\section{Applications and complements}
\label{sect:related}
Here we show that our methods allow to obtain quickly some old results related to 
the Adler--Kostant--Symes (AKS) theory and point out directions for further research.
\subsection{Relations with AKS theory}    
\label{subs:aks}
Let $\q=\ah_1\oplus\ah_2$ be an  arbitrary splitting. Then $\q^*=\ah_1^*\oplus\ah_2^*$, where 
$\ah_1^*=\ah_2^\perp$ and $\ah_2^*=\ah_1^\perp$. Recall that the elements of $\gS(\q)$ are polynomial functions on $\q^*$ and if $\cF\in \gS^d(\q)$, then $\cF=\sum_{i=0}^d\cF_{i,d-i}$, where 
$\cF_{i,d-i}\in \gS^i(\ah_1)\otimes\gS^{d-i}(\ah_2)$. The bi-homogeneous component $\cF_{0,d}$
(resp. $\cF_{d,0}$) can be identified with the restriction $\cF\vert_{\ah_2^*}$ (resp. $\cF\vert_{\ah_1^*}$).
Let $\ca_i$ denote the restriction of $\gS(\q)^\q$ to $\ah_i^*$. 

\begin{thm}[{cf.~\cite[Theorem\,2.2]{S}} ]     
\label{thm:aks1}
For $i=1,2$, the subalgebra $\ca_i\subset \gS(\ah_i)$ is Poisson-commutative 
\wrt\ the Lie--Poisson bracket on $\ah_i^*$.
\end{thm}
\begin{proof}
By our general approach, we have the In\"on\"u--Wigner contractions $\q_{(0)}=\ah_1\ltimes\ah_2^{\sf ab}$ and $\q_{(\infty)}=\ah_2\ltimes\ah_1^{\sf ab}$ together with their Lie--Poisson brackets
$\{\ ,\,\}_0$ and $\{\ ,\,\}_\infty$. By Proposition~\ref{prop:Z-x}, the bi-homogeneous components of all
$\cF\in\gS(\q)^\q$ commute \wrt\ all brackets in the family $\{\ ,\,\}_0+t\{\ ,\,\}_\infty$ ($t\in\BP$).
In particular, being bi-homogeneous components, the elements of $\ca_1$ commute \wrt\ $\{\ ,\,\}_0$.
Since $\ca_1\subset \gS(\ah_1)$ and $\ah_1$ is a Lie subalgebra of $\q_{(0)}$, we conclude that
the elements of $\ca_1$ commute \wrt\ the Lie--Poisson bracket on $\ah_1^*$.
The argument for $\ca_2$ is similar.
\end{proof}

A slightly different result occurs for the involutions of a reductive Lie algebra $\g$. Let $\sigma$ be an
involution of $\g$ and $\g=\g_0\oplus\g_1$. Then there is a parabolic subalgebra $\p\subset\g$ such that
$\g=\g_0\oplus\h$ for some $\h$ between $[\p,\p]$ and $\p$. More precisely, if $\ce\subset\g_1$ is a
Cartan subspace, then $\g^\ce$ is a Levi subalgebra of $\p$. (Hence $\p=\be$ if and only if $\sigma$ is
$S$-regular.) Actually, the sum $\g=\g_0\oplus\h$ is the complexification of the Iwasawa decomposition 
for the real form of $\g$ corresponding to $\sigma$ (over $\bbk=\BC$).
One has the linear isomorphism $\psi: \g_1\to \h^*$ such that $\psi(x_1)(y)=(x_1,y)$ for $x_1\in \g_1$ and
$y\in\h$. For $f\in \bbk[\g_1]^{G_0}$, we define $H_f\in \gS(\h)$ by the formula $H_f(\xi)=f(\psi^{-1}(\xi))$
for $\xi\in\h^*$.
\begin{thm}[{cf.~\cite[Sect.\,3.1]{GW}}]    \label{thm:aks2}
$\{H_f \mid f\in \bbk[\g_1]^{G_0}\}$ is a\/ \PC\ subalgebra of\/ $\gS(\h)$. 
\end{thm}
\begin{proof}   Consider the restriction $p: \bbk[\g]^G\to \bbk[\g_1]^{G_0}$, $\cF\mapsto \cF\vert_{\g_1}$.
If $f=p(\cF)$ and $\deg\cF=d$, then $H_f$ is nothing but the bi-homogeneous component 
$\cF_{0,d}$ \wrt\ the sum $\g=\g_0\oplus\h$. Therefore, Theorem~\ref{thm:aks1} shows that
$\{H_f \mid f\in p(\bbk[\g]^{G})\}$ is a\/ \PC\ subalgebra of\/ $\gS(\h)$. By~\cite[Prop.\,7.1]{py23}, the
extension $p(\bbk[\g]^{G})\subset \bbk[\g_1]^{G_0}$ is finite. It is also easily seen that if 
$\ca\subset \gS(\h)$ is \PC\ and $\ca\subset\tilde\ca$ is an algebraic extension, then 
$\tilde\ca\subset\gS(\h)$ is \PC\ as well.
\end{proof}

\subsection{Some open problems and directions for further research}
\label{subs:further}
We already know many non-degenerate splittings $\g=\h\oplus\rr$ such that $\gZ_{\lg\h,\rr\rg}$ is a
(large) polynomial ring. Furthermore, in each such good case, a Hilbert basis of $\gZ_{\lg\h,\rr\rg}$ is 
explicitly described.

\begin{qtn}
When is it true that $\gZ_{\lg\h,\rr\rg}$ is also a maximal \PC\ subalgebra of $\gS(\g)$ with respect to inclusion? 
\end{qtn}
So far, the maximality is proved for $\gZ_{\lg\be,\ut_-\rg}$ (any $\g$) and $\gZ_{\lg\g_0,\be\rg}$ ($\g=\sln$), see
Introduction.

\begin{prob}
In the good cases, the inclusion $\gZ_{\lg\h,\rr\rg}\subset \gS(\g)$ yields the morphism
\[
    \tau: \g^*\to \spe\gZ_{\lg\h,\rr\rg} \simeq \BA^{\bb(\g)}.
\]
Prove that $\tau$ is flat, i.e., a (any) Hilbert basis of\/ $\gZ_{\lg\h,\rr\rg}$ forms a regular sequence in 
$\gS(\g)$.
\end{prob}
More generally, if $\gZ$ is a large polynomial \PC\ subalgebra, then the same problem applies to
$\tau:\g^*\to \spe \gZ\simeq \BA^{\bb(\g)}$. For the Mishchenko--Fomenko algebras $\gZ_\gamma$
with $\gamma\in\g^*_{\sf reg}$, an affirmative answer is given by Moreau~\cite{Ann}.

\begin{prob}
Quantise large \PC\ subalgebras of the form  $\gZ_{\lg\h,\rr\rg}$, i.e., construct commutative subalgebras
$\widetilde\gZ_{\lg\h,\rr\rg}$ of the enveloping algebra $\eus U(\g)$ such that 
$\gr\widetilde\gZ_{\lg\h,\rr\rg}=\gZ_{\lg\h,\rr\rg}$.
\end{prob}
A quantisation of $\gZ_{\langle\be,\ut_-\rangle}$ is recently obtained in \cite{flower}, whereas the Mishchenko--Fomenko algebras $\gZ_\gamma$ are quantised by Rybnikov~\cite{Ry}.

\end{document}